\documentclass[leqno]{amsart}

\usepackage{stmaryrd,graphicx}
\usepackage{amssymb,mathrsfs,amsmath,amscd,amsthm,color}
\usepackage{float}
\usepackage{mathabx}
\usepackage[all,cmtip]{xy}
\DeclareMathAlphabet{\mathpzc}{OT1}{pzc}{m}{it}
\usepackage{amsfonts,latexsym,wasysym}

\newcommand{\jx}{J(X)}
\newcommand{\jnx}{J_n(X)}

\renewcommand{\bigast}{\Asterisk}

\newcommand{\ui}{I}

\newcommand{\wt}{\widetilde}

\newcommand{\pioneh}{\pi_{1}(\bbh,b_0)}

\newcommand{\mca}{\mathcal{A}}

\newcommand{\mcc}{\mathcal{C}}

\newcommand{\mcd}{\mathcal{D}}

\newcommand{\mci}{\mathcal{I}}

\newcommand{\scru}{\mathscr{U}}

\newcommand{\bba}{\mathbb{A}}

\newcommand{\bbg}{\mathbb{G}}
\newcommand{\bbh}{\mathbb{H}}
\newcommand{\bbn}{\mathbb{N}}
\newcommand{\bbq}{\mathbb{Q}}
\newcommand{\bbr}{\mathbb{R}}

\newcommand{\bbt}{\mathbb{T}}
\newcommand{\bbz}{\mathbb{Z}}

\newcommand{\ov}{\overline}

\newtheorem{theorem}{Theorem}[section]
\newtheorem{lemma}[theorem]{Lemma}
\newtheorem{proposition}[theorem]{Proposition}
\newtheorem{corollary}[theorem]{Corollary}
\theoremstyle{definition}\newtheorem{definition}[theorem]{Definition}
\newtheorem{example}[theorem]{Example}

\newtheorem{remark}[theorem]{Remark}
\newtheorem{problem}[theorem]{Problem}

\begin{document}
\title[Infinitary commutativity and fundamental groups]{Infinitary commutativity and fundamental groups of topological monoids}
\author[J. Brazas]{Jeremy Brazas}
\address{West Chester University\\ Department of Mathematics\\
West Chester, PA 19383, USA}
\email{jbrazas@wcupa.edu}

\author[P. Gillespie]{Patrick Gillespie}
\address{West Chester University\\ Department of Mathematics\\
West Chester, PA 19383, USA}
\email{pg915111@wcupa.edu}


\date{\today}

\begin{abstract}
The well-known Eckmann-Hilton Principle may be applied to prove that fundamental groups of $H$-spaces are commutative. In this paper, we identify an infinitary analogue of the Eckmann-Hilton Principle that applies to fundamental groups of all topological monoids and slightly more general objects called pre-$\Delta$-monoids. In particular, we show that every pre-$\Delta$-monoid $M$ is  ``transfinitely $\pi_1$-commutative" in the sense that permutation of the factors of any infinite loop-concatenation indexed by a countably infinite order and based at the identity $e\in M$ is a homotopy invariant action. We also give a detailed account of fundamental groups of James reduced products and apply transfinite $\pi_1$-commutativity to make several computations.
\end{abstract}

\maketitle

\section{Introduction}

The Eckmann-Hilton Principle \cite{EckHilt} states that if a set $M$ is equipped with two unital binary operations $\ast$ and $\cdot$ satisfying the distributive law $(a\cdot b)\ast (c\cdot d)=(a\ast c)\cdot (b\ast d)$, then the operations $\ast$ and $\cdot$ agree and are associative and commutative. Applying this principle to the fundamental group $\pi_1(X,e)$ of any $H$-space $(X,e)$, it follows that $\pi_1(X,e)$ is abelian \cite[\textsection 3C, Exercise 5]{Hatcher}. Since every loop space $\Omega(X,e)$ is an $H$-space, one may apply this reasoning to confirm that all higher homotopy groups are abelian.

Infinitary operations akin to infinite sums and products in analysis arise naturally in the context of fundamental groups and are often highly non-commutative. These operations arise from the ability to form an infinite loop concatenation $\prod_{n=1}^{\infty}\alpha_n$ (with order type $\omega$) and a transfinite loop-concatenation $\prod_{\tau}\alpha_n$ (with dense order type) from a shrinking sequence of loops $\{\alpha_n\}$ based at a point. Computations of singular homology groups of spaces such as the Hawaiian earring \cite{EKH1ofHE} and other ``wild" spaces \cite{Edasingularonedim16,EdaFischer,HHwildhomology,KarimovRepovs} that admit non-trivial infinitary $\pi_1$-operations typically require abstract infinite abelian group theory \cite{Fuchs} since $H_1(X)$ is only the ``finite" abelianization of the group $\pi_1(X,x)$. For instance, if one projects the homotopy class of a product $\left[\prod_{n=1}^{\infty}\alpha_n\right]\in \pi_1(X,x)$ into $H_1(X)$, one may only decompose the resulting homology class into a finite factorization and permute those finitely many factors without changing the homology class. In general, permuting infinitely many factors will change the homology class.

Remarkably, the difference between finitary and infinitary commutativity seems to disappear in higher dimensions. The work of Eda-Kawamura \cite{EK00higher} on the higher homotopy groups of the $n$-dimensional Hawaiian earrings ($n\geq 2$), suggests that higher homotopy groups are always ``infinitely commutative" in the sense that if we are given an infinite/transfinite product of based $n$-loops, infinite permutation of these factors is a homotopy invariant action. In this paper, we show that this strong version of commutativity occurs in the fundamental group of any topological monoid (and slightly weaker objects called pre-$\Delta$-monoids). We define the point-wise property of a space $X$ being \textit{transfinitely $\pi_1$-commutative} at a point $x \in X$ (Definition \ref{deftransfinitecommut}) and show that every pre-$\Delta$-monoid is transfinitely $\pi_1$-commutative at its identity element (Theorem \ref{monoidtransfinitecommutethm}). We consider the infinite shuffle argument behind this result an infinitary analogue of the Eckmann-Hilton Principle. A crucial part of our argument is maintaining control over the sizes of the homotopies being used at each step.

We apply transfinite $\pi_1$-commutativity by investigating a fundamental construction, namely, the James reduced product $J(X)$ of a based space $(X,e)$. This construction, introduced by I.M. James \cite{James}, is important in classical homotopy theory since $J(X)$ is homotopy equivalent to $\Omega(\Sigma X)$ when $X$ is a connected CW-complex. Applying the fundamental group to this equivalence gives $\pi_1(J(X),e)\cong \pi_2(\Sigma X)\cong H_2(\Sigma X)\cong H_1(X)$, confirming that $\pi_1(J(X),e)$ is the abelianization of $\pi_1(X,e)$ for CW-complex $X$. In the second half of this paper, we give a detailed study of $\pi_1(J(X),e)$ in a much more general setting. Our main technical achievement in Section \ref{sectionjrproduct} is the proof that the inclusion $\sigma:X\to J(X)$ induces a surjection on fundamental groups for any path-connected Hausdorff space (Theorem \ref{surjectivitytheorem}). In particular, our analysis applies to spaces, such as the Hawaiian earring, in which there might be infinitary $\pi_1$-products at the basepoint. Combining Theorem \ref{surjectivitytheorem} with transfinite $\pi_1$-commutativity allows us to compute $\pi_1(J(X),e)$ for several important examples, including when $X$ is the Hawaiian earring, double Hawaiian earring, harmonic archipelago, and Griffiths twin cone. Our treatment of these examples illustrates that $\pi_1(J(X),e)$ appears to behave as a kind of \textit{transfinite abelianization} of $\pi_1(X,e)$ at $e\in X$, a notion that we formalize in Remark \ref{infinitaryabelianization}.

The remainder of this paper is structured as follows. In Section \ref{sectionprelim}, we settle notation and define the notion of a pre-$\Delta$-monoid (Definition \ref{predeltamonoiddef}). This generalization of topological monoids is critical for our study of James reduced products since $J(X)$ is always a pre-$\Delta$-monoid but can fail to be an actual topological monoid (Remark \ref{fretopmonoidremark}). In Section \ref{transfinitesection}, we define and study what it means for a space to be transfinitely $\pi_1$-commutative at a point. In particular, we prove a useful characterization of this property in terms of factorization through the Specker group $\bbz^{\bbn}$ (Theorem \ref{transfinitecharacteriationtheorem}). In Section \ref{sectionmonoids}, we prove that all pre-$\Delta$-monoids are transfinitely $\pi_1$-commutative at their identity element using an infinitary Eckmann-Hilton-type argument (Theorem \ref{monoidtransfinitecommutethm}). In Section \ref{sectionjrproduct}, we focus our attention on the fundamental group of the James reduced product $J(X)$ for a path-connected Hausdorff space $X$ with basepoint $e\in X$. Since our results apply to such a broad class of spaces, we begin this section with a detailed study of the topology of $J(X)$. Section \ref{sectionjrproduct} concludes with several computations of the group $\pi_1(J(X),e)$, the examples being chosen to illustrate a variety of techniques and applications of infinitary commutativity.

\section{Preliminaries and Notation}\label{sectionprelim}

For spaces $X,Y$, let $Y^X$ denote the space of maps $f:X\to Y$ with the compact-open topology generated by subbasic sets $\langle K,U\rangle=\{f\mid f(K)\subseteq U\}$ where $K\subseteq X$ is compact and $U\subseteq Y$ is open. If $A\subseteq X$, $B\subseteq Y$, then $(Y,B)^{(X,A)}\subseteq Y^X$ will denote the subspace of relative maps satisfying $f(A)\subseteq B$. If $A=\{x_0\}$ and $B=\{y_0\}$ contain only basepoints, we may simply write $(Y,y_0)^{(X,x_0)}$. The constant function $X\to Y$ at $y_0$ is denoted $c_{y_0}$. If $f:(X,x)\to (Y,y)$ is a based map, then $f_{\#}:\pi_1(X,x)\to \pi_1(Y,y)$ denotes the homomorphism induced on the fundamental group.

\begin{definition}
A sequence $\{f_n\}$ in $Y^X$ is \textit{null at }$y\in Y$ if for every open neighborhood $U$ of $y$, there is an $N\in\bbn$ such that $Im(f_n)\subseteq U$ for all $n\geq N$, i.e. if $\{f_n\}\to c_y$ in the compact-open topology. In particular, we refer to $\{f_n\}$ as a \textit{null-sequence}.
\end{definition}

A \textit{path} in a topological space $X$ is a continuous function $I\to X$ where $I=[0,1]$ is the closed unit interval. If $\alpha_1,\alpha_2,\dots,\alpha_n:I\to X$ is a sequence of paths satisfying $\alpha_i(1)=\alpha_{i+1}(0)$, we write $\prod_{n=1}^{m}\alpha_n$ or $\alpha_1\cdot\alpha_2\cdots\alpha_n$ for the n-fold \textit{concatenation} defined to be $\alpha_i$ on the interval $\left[\frac{i-1}{n},\frac{i}{n}\right]$. We write $\alpha^{-}(t)=\alpha(1-t)$ for the \textit{reverse} of a path $\alpha$. We write $\Omega(X,x_0)$ for the based loop space $(X,x_0)^{(I,\partial I)}$. If $[a,b],[c,d]\subseteq \ui$ and $\alpha:[a,b]\to X$, $\beta:[c,d]\to X$ are maps, we write $\alpha\equiv\beta$ if $\alpha=\beta\circ \lambda$ for some increasing homeomorphism $\lambda: [a,b]\to [c,d]$; if $\lambda$ is linear and if it does not create confusion, we will identify $\alpha$ and $\beta$. Throughout this paper, any space in which paths are considered is assumed to be path-connected. 

A \textit{monoid} is a set $M$ equipped a monoid operation, i.e. an associative, binary operation $\mu:M\times M\to M$, $\mu(a,b)=a \ast b$, which has an identity element $e\in M$. We utilize the following notation:
\begin{itemize}
\item If $f,g:X\to M$ are maps, let $f\ast g:X\to M$ denote the function $(f\ast g)(x)=f(x)\ast g(x)$.
\item If $f:X\to M$ is a map and $a\in M$, we let $a\ast f$ and $f\ast a$ denote the functions $X\to M$ given by $x\mapsto a\ast f(x)$ and $x\mapsto f(x)\ast a$ respectively.
\item More generally, we write $\bigast_{i=1}^{n}a_i$ to denote iterated products of elements $a_i\in M$. We also use this notation if one or more of the $a_i$ are functions $X\to M$ from a fixed set $X$, e.g. if $a_1,a_3\in M$ and $a_2,a_4:X\to M$ are functions, then $\bigast_{i=1}^{4}a_i:X\to M$ is the function $x\mapsto a_1a_2(x)a_3a_4(x)$.
\item If $A_1,A_2,\dots, A_n\subseteq M$, then $\bigast_{i=1}^{n}A_i$ denotes the product $\mu_n\left(\prod_{i=1}^{n}A_i\right)=\left\{\bigast_{i=1}^{n}a_i\mid a_i\in A_i\right\}$ where $\mu_n:M^n\to M$, $\mu_n(a_1,a_2,\dots,a_n)=\bigast_{i=1}^{n}a_i$ is the $n$-ary multiplication function. In particular, we write $A\ast B$ for $\{a\ast b\mid a\in A,b\in B\}$.

\end{itemize}
If $M$ is also equipped with a topology such that $\mu$ is continuous, then we call $M$ a \textit{topological monoid}. Although concatenation $(\alpha,\beta)\mapsto \alpha\cdot\beta$ is a continuous binary operation on $\Omega(X,x_0)$, it is only associative and unital up to homotopy. Despite this inconvenience, one may employ the homotopy equivalent Moore loop space $\Omega^{\ast}(X,x_0)$, which is a genuine topological monoid. We refer to \cite{WhiteheadEOH} for the details and remark that the homotopy equivalence $\Omega(X,x_0)\to \Omega^{\ast}(X,x_0)$ in \cite[Corollary 2.19]{WhiteheadEOH} holds for arbitrary path-connected spaces $X$.

Since we work in a very general context, there will be some instances when the operation $M\times M\to M$ is not jointly continuous, but is ``continuous enough" to ensure that all of our results apply. To formalize this idea, we recall the notion of a $\Delta$-generated space.

\begin{definition}
A topological space $X$ is $\Delta$\textit{-generated} if $U\subseteq X$ is open if and only if $\alpha^{-1}(U)$ is open in $\ui$ for all paths $\alpha:\ui\to X$.
\end{definition}

Note that a space $X$ is $\Delta$-generated if and only if it is the topological quotient of a disjoint sum of copies of $\ui$. Hence, every Peano continuum (connected, locally path-connected, compact metric space) is $\Delta$-generated and every $\Delta$-generated space is a locally path-connected $k$-space. For any space $X$, let $\Delta(X)$ be the space with the same underlying set as $X$ but with the topology generated by the sets $\alpha^{-1}(U)$ for all paths $\alpha:\ui\to X$. Certainly, this new topology is finer than the original topology on $X$ and is equal to the topology on $X$ if and only if $X$ is already $\Delta$-generated. Moreover, $\Delta(X)$ is $\Delta$-generated and is characterized by the universal property that if $f:D\to X$ is a map from a $\Delta$-generated space $D$, then $f:D\to \Delta(X)$ is also continuous. In particular, the identity function $\Delta(X)\to X$ is a weak homotopy equivalence. We refer to \cite{CSWdiff,FRdirectedhomotopy} for more on $\Delta$-generated spaces.

\begin{definition}\label{predeltamonoiddef}
A \textit{pre-$\Delta$-monoid} consists of a space $M$ and a monoid operation $\mu:M\times M\to M$, $\mu(a,b)=a\ast b$ on $M$ with the property that for any paths $\alpha,\beta:\ui\to M$, the product path $\alpha\ast\beta:\ui\to M$ is continuous. If $M$ is also a group, then we refer to $M$ as a \textit{pre-$\Delta$-group}. If $M$ is $\Delta$-generated, then we call $M$ a \textit{$\Delta$-monoid}.
\end{definition}

Every topological monoid is a pre-$\Delta$-monoid. The following proposition illustrates the sense in which a pre-$\Delta$-monoid is ``almost" a topological monoid. We are particularly motivated to develop results for pre-$\Delta$-monoids since the \textit{James reduced product} $\jx$ studied in Section \ref{sectionjrproduct} is always a pre-$\Delta$-monoid for Hausdorff $X$ but may not be a genuine topological monoid.

\begin{proposition}\label{predeltamoncharprop}
For any monoid $M$ with topology and operation $\mu:M^2\to M$, the following are equivalent:
\begin{enumerate}
\item $M$ is a pre-$\Delta$-monoid,
\item $\Delta(M)$ is a $\Delta$-monoid,
\item the $k$-ary operation $\mu_k:\Delta(M^k)\to M$, $\mu_{k}(a_1,a_2,\dots a_k)=\bigast_{i=1}^{k}a_i$ is continuous for all $k\geq 2$,
\item The operation $M^D\times M^D\to M^D$, $(f,g)\mapsto f\ast g$ is well-defined for any $\Delta$-generated space $D$.
\end{enumerate}
\end{proposition}

\begin{proof}
(1) $\Rightarrow$ (2) If $M$ is a pre-$\Delta$-monoid and $\alpha,\beta:I\to \Delta(M)$ are paths, then $\alpha,\beta:I\to M$ are also continuous. Thus $\alpha\ast\beta:I\to M$ is continuous by assumption. By the universal property of $\Delta(M)$, $\alpha\ast\beta:I\to \Delta(M)$ is continuous.

(2) $\Rightarrow$ (3) Suppose $\Delta(M)$ is a $\Delta$-monoid and let $U\subseteq M$ be open. Let $f=(\alpha_1,\alpha_2,\dots,\alpha_k):I\to M^k$ be any path. Then $\alpha_i:I\to \Delta(M)$ is continuous for each $i$ and we have that $\alpha=\mu_{k}\circ f=\bigast_{i=1}^{k}\alpha_i:I\to \Delta(M)$ is continuous by assumption. Since $\Delta(M)$ has a finer topology than $M$, $\alpha:I\to M$ is continuous. Thus $\alpha^{-1}(U)=f^{-1}(\mu_{k}^{-1}(U))$ is open in $\ui$. This shows that $\mu_{k}^{-1}(U)$ is in the topology of $\Delta(M^k)$, proving the continuity of $\mu_k$.

(3) $\Rightarrow$ (4) Suppose $\mu:\Delta(M^2)\to M$ is continuous, $D$ is $\Delta$-generated, and $f,g:D\to M$ are maps. Then $(f,g):D\to M^2$ is continuous and, since $D$ is $\Delta$-generated, so is $(f,g):D\to \Delta(M^2)$. Thus $f\ast g=\mu\circ (f,g)$ is continuous as a function $D\to M$.

(4) $\Rightarrow$ (1) is obvious since $\ui$ is $\Delta$-generated.
\end{proof}

Within the category of $\Delta$-generated spaces, the categorical product of two $\Delta$-generated spaces $X,Y$ is $\Delta(X\times Y)$. Hence, condition (3) in Proposition \ref{predeltamoncharprop} indicates that $M$ is a pre-$\Delta$-monoid if and only if $\Delta(M)$ is a monoid-object in the category of $\Delta$-generated spaces.

\section{Transfinite commutativity in fundamental groups}\label{transfinitesection}

\subsection{The Hawaiian earring and infinitary operations on loops}

To define the notion of transfinite $\pi_1$-commutativity, we employ the fundamental group of the Hawaiian earring. Let $C_n\subset \bbr^2$ denote the circle of radius $\frac{1}{n}$ centered at $\left(\frac{1}{n},0\right)$ and $\bbh=\bigcup_{n\in\bbn}C_n$ be the \textit{Hawaiian earring} with basepoint $b_0=(0,0)$. We define some important loops in $\bbh$ as follows:
\begin{itemize}
\item For each $n\in\bbn$, let $\ell_n\in \Omega(C_n,b_0)$ be the canonical counterclockwise loop traversing the circle $C_n$.
\item Let $\ell_{\infty}\in \Omega(\bbh,b_0)$ denote the loop defined as $\ell_{n}$ on the interval $\left[\frac{n-1}{n},\frac{n}{n+1}\right]$ and $\ell_{\infty}(1)=b_0$.
\item Let $\mcc\subseteq \ui$ be the middle third Cantor set. Write $\ui\backslash \mcc=\bigcup_{n\geq 1}\bigcup_{k=1}^{2^{n-1}} I_{n}^{k}$ where $I_{n}^{k}$ is an open interval of length $\frac{1}{3^n}$ and, for fixed $n$, the sets $I_{n}^{k}$ are indexed by their natural ordering in $\ui$. Let $\ell_{\tau}\in \Omega(\bbh,b_0)$ be the loop defined so that $\ell_{\tau}(\mcc)=b_0$ and $\ell_{\tau}:=\ell_{2^{n-1}+k-1}$ on $\overline{I_{n}^{k}}$ (see Figure \ref{elltaufigure}).
\end{itemize}

\begin{figure}[H]
\centering \includegraphics[height=0.6in]{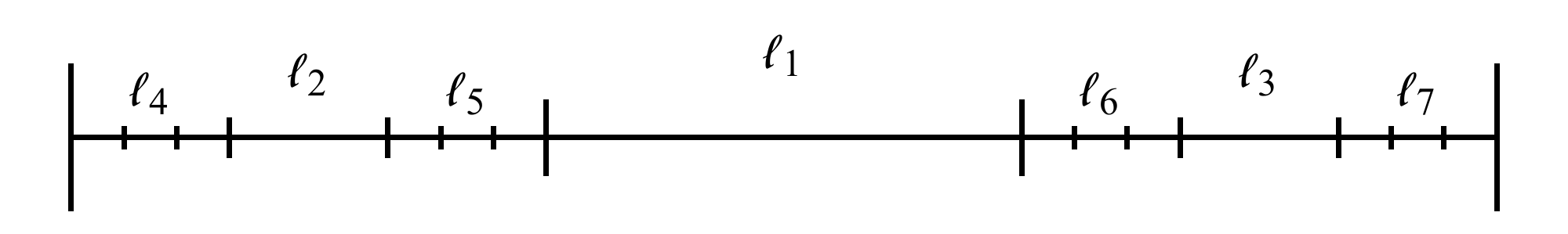}
\caption{\label{elltaufigure}The tranfinite concatenation loop $\ell_{\tau}$}
\end{figure}

The fundamental group $\pioneh$ is uncountable and not free. However, $\pioneh$ is locally free and naturally isomorphic to a subgroup of an inverse limit of free groups. For any set $A\subseteq \bbn$, let $\bbh_{A}=\bigcup_{n\in A}C_n$. As special cases, let $\bbh_{\geq n}=\bigcup_{m\geq n}C_m$ be the smaller homeomorphic copies of $\bbh$ and let $\bbh_{\leq n}=\bigcup_{m=1}^{n}C_n$ be the wedge of the first $n$-circles so that $\pi_1(\bbh_{\leq n},b_0)=F_n$ is the group freely generated by the elements $[\ell_1],[\ell_2],\dots,[\ell_n]$. The retractions $r_{n+1,n}:\bbh_{\leq n+1}\to \bbh_{\leq n}$ collapsing $C_{n+1}$ to $b_0$ induce an inverse sequence
\[\dots \to F_{n+1}\to F_n\to \dots \to F_2\to F_1\]on fundamental groups, in which $F_{n+1}\to F_n$ deletes the letter $[\ell_{n+1}]$ from a given word. The inverse limit $\check{\pi}_{1}(\bbh,b_0)=\varprojlim_{n}F_n$ is the first shape homotopy group \cite{MS82}. The retractions $r_n:\bbh\to \bbh_{\leq n}$, which collapse $\bbh_{\geq n+1}$ to $b_0$, induce a canonical homomorphism \[\psi:\pioneh\to \check{\pi}_{1}(\bbh,b_0)\text{ where }\psi([\alpha])=([r_1\circ\alpha],[r_2\circ\alpha],\dots).\]
It is well known that $\psi$ is injective \cite{MM} and more generally that $\pi_1$-injectivity into the first shape homotopy group holds for all one-dimensional metric spaces \cite{EK98oned}. Thus a homotopy class $[\alpha]\in\pi_1(\bbh,b_0)$ is trivial if and only if for every $n\in \bbn$, the projection $(r_n)_{\#}([\alpha])\in F_n$ reduces to the trivial word as a word in the letters $[\ell_1],[\ell_2],\dots,[\ell_n]$. Based on the injectivity of $\psi$, we also note that for every finite set $F\subseteq \bbn$, the inclusions $\bbh_F\to\bbh$ and $\bbh_{\bbn\backslash F}\to \bbh$ induce an isomorphism $\pi_1(\bbh_{F},b_0)\ast \pi_1(\bbh_{\bbn\backslash F},b_0)\to \pi_1(\bbh,b_0)$. Hence any $[\alpha]\in\pioneh$ factors uniquely as a finite product of homotopy classes alternating between $\bbh_F$ and $\bbh_{\bbn\backslash F}$.

\begin{remark}
If $f:(\bbh,b_0)\to (X,x)$ is a map, then the sequence $\{f\circ\ell_n\}$ is null at $x$. Conversely, if $\{\alpha_n\}$ is a null-sequence of loops based at $x$, then we may define a continuous map $f:\bbh\to X$ by $f\circ \ell_n=\alpha_n$. Hence, null-sequences of loops based at $x\in X$ are in bijective correspondence with maps $(\bbh,b_0)\to (X,x)$. 
\end{remark}

\begin{definition}
Suppose $\alpha_n\in \Omega(X,x)$ is a null-sequence and $f:\bbh\to X$ is the map with $f\circ \ell_n=\alpha_n$. 
\begin{itemize}
\item We write $\prod_{n=1}^{\infty}\alpha_n$ for the loop $f\circ\ell_{\infty}$, which we call the \textit{infinite concatenation} of the sequence $\{\alpha_n\}$.
\item We write $\prod_{\tau}\alpha_n$ for the loop $f\circ \ell_{\tau}$, which we call the \textit{transfinite concatenation} of the sequence $\{\alpha_n\}$.
\end{itemize}
\end{definition} 

%

Let $\bbt=\prod_{n\in\bbn}S^1$ be the infinite torus with identity element/basepoint $x_0$ and let $\eta:\bbh\to \bbt$ be the canonical embedding onto the subspace $\bigvee_{n=1}^{\infty} S^1$ of $\bbt$. The fundamental group $\pi_1(\bbt,x_0)$ may be identified with the Specker group $\bbz^{\bbn}$ where $[\eta\circ\ell_n]$ is identified with the unit vector $\mathbf{e}_n\in \bbz^{\bbn}$, which has $1$ in the $n$-th coordinate and $0$'s elsewhere. It is well-known that the induced homomorphism $\eta_{\#}:\pioneh\to\pi_1(\bbt,x_0)$ is surjective. In \cite{CC,Edafreesigmaproducts}, $\bbz^{\bbn}$ is described as a ``strong abelianization" of $\pioneh$. Our more general approach to infintite commutativity is founded upon the difference between the highly non-commutative group $\pioneh$ and the highly commutative group $\bbz^{\bbn}$.

The proof of the following lemma is sketched in \cite[Example 3.10]{BFTestMap}. Since it is important for the proof of Theorem \ref{transfinitecharacteriationtheorem}, we give a more detailed proof. It states that any non-trivial element $[\alpha]\in\pioneh$ where $\alpha$ has winding number $0$ around $C_n$ for all $n\in\bbn$ factors as the infinite concatenation of a null-sequence of commutators. 

\begin{lemma}\label{infinitecommutatorlemma}
If $1\neq [\alpha]\in \ker(\eta_{\#})$, then there is a null-sequence $\gamma_n\in\Omega(\bbh,b_0)$ such that $[\alpha]=\left[ \prod_{n=1}^{\infty}\gamma_n\right]$ and $\gamma_n\equiv\prod_{i=1}^{k_n}\left(\alpha_{n,i}\cdot\beta_{n,i}\cdot \alpha_{n,i}^{-}\cdot\beta_{n,i}^{-}\right)$ for loops $\alpha_{n,i},\beta_{n,i}\in \Omega(\bbh_{\geq n},b_0)$ and integers $k_n\in\bbn$.
\end{lemma}

\begin{proof}
Let $G=\pioneh$ and $G_n=\pi_1(\bbh_{\geq n},b_0)$ for $n\geq 2$ and let $[K,K]$ denote the commutator subgroup of a group $K$. Consider a non-contractible loop $\alpha\in\Omega(\bbh,b_0)$ such that $[\alpha]\in \ker(\eta_{\#})$. Set $\gamma_0=c_{b_0}$ and $\beta_0=\alpha$ so that $[\alpha]=[\gamma_0][\beta_0]$. Proceeding by induction, suppose that we have constructed loops $\gamma_i\in\Omega(\bbh_{\geq i},b_0)$, $0\leq i\leq n-1$ and $\beta_{n-1}\in \Omega(\bbh_{\geq n},b_0)$ such that
\begin{itemize}
\item $[\alpha]=\left(\prod_{i=0}^{n-1}[\gamma_i]\right)[\beta_{n-1}]$,
\item $[\gamma_i]\in [G_i,G_i]$,
\item $[\beta_{n-1}]\in \ker(\eta_{\#})$.
\end{itemize}
Since $[\beta_{n-1}]\in \pi_1(\bbh_{\geq n},b_0)\cap \ker(\eta_{\#})$, we have factorization $[\beta_{n-1}]=\prod_{i=1}^{m}\left([\delta_i][\ell_{n}]^{\epsilon_i}\right)$ where the loop $\delta_i$ has image in $\bbh_{\geq n+1}$ and $\sum_{i=1}^{m}\epsilon_i=0$. Since $\beta_{n-1}$ and $\beta_{n}=\prod_{i=1}^{m}\delta_i$ are homologous in $\bbh_{\geq n}$, there is a loop $\gamma_{n}\in\Omega(\bbh_{\geq n},b_0)$ such that $[\gamma_n]\in [G_n,G_n]$ and $[\beta_{n-1}]=[\gamma_{n}][\beta_{n}]$. Since $[\beta_{n-1}]\in \ker(\eta_{\#})$ and $[\gamma_{n}]\in [G_n,G_n]\leq [G,G]\leq \ker (\eta_{\#})$, we have $[\beta_{n}]\in \ker(\eta_{\#})$. This completes the induction.

The induction provides a null sequence $\{\gamma_n\}$ such that $[\gamma_n]\in[G_n,G_n]$ and $[\alpha]=\left(\prod_{i=1}^{n}[\gamma_i]\right)[\beta_n]$ for all $n\in\bbn$. Notice that $\left[\prod_{n=1}^{\infty}\gamma_n\right]^{-1}[\alpha]$ is represented by each loop in the null-sequence $\left\{\left(\prod_{i=n+1}^{\infty}\gamma_i\right)^{-}\cdot\beta_n\right\}_{n\in\bbn}$. Thus, since $\bbh$ is homotopically Hausdorff at $b_0$, we have $\left[\prod_{n=1}^{\infty}\gamma_n\right]^{-1}[\alpha]=1$ in $\pioneh$. We conclude that $[\alpha]=\left[\prod_{n=1}^{\infty}\gamma_n\right]$ where $\gamma_n$ has the form described in the statement of the lemma.
\end{proof}

\subsection{Defining and characterizing transfinite $\pi_1$-commutativity}

Since $\ell_{\infty}$ and $\ell_{\tau}$ are fixed, permuting the terms of a null-sequence $\{\alpha_n\}$ will typically change the homotopy class of both $\prod_{n=1}^{\infty}\alpha_n$ and $\prod_{\tau}\alpha_n$. The following definition describes when such permutations are homotopy invariant.

\begin{definition}\label{deftransfinitecommut}
A space $X$ is 
\begin{enumerate}
\item \textit{infinitely $\pi_1$-commutative at} $x\in X$ if for every null sequence $\alpha_n\in \Omega(X,x)$ and bijection $\phi:\bbn\to\bbn$, we have $[\prod_{n=1}^{\infty}\alpha_n]=[\prod_{n=1}^{\infty}\alpha_{\phi(n)}]$ in $\pi_1(X,x)$.
\item \textit{transfinitely $\pi_1$-commutative at} $x\in X$ if for every null sequence $\alpha_n\in \Omega(X,x)$ and bijection $\phi:\bbn\to\bbn$, we have $[\prod_{\tau}\alpha_n]=[\prod_{\tau}\alpha_{\phi(n)}]$ in $\pi_1(X,x)$.
\end{enumerate}
\end{definition}

The Rearrangement Theorem in Calculus says that a convergent infinite sum $\sum_{n=1}^{\infty}a_n$ of real numbers is absolutely convergent if and only if for every bijection $\phi:\bbn\to\bbn$, we have $\sum_{n=1}^{\infty}a_n=\sum_{n=1}^{\infty}a_{\phi(n)}$. The property of being infinitely $\pi_1$-commutative is analogous.

\begin{remark}\label{remarktransfcomimpliesabelian}
If $X$ is infinitely or transfinitely $\pi_1$-commutative at $x\in X$, then $\pi_1(X,x)$ is commutative in the usual sense. For example, if $X$ is transfinitely $\pi_1$-commutative at $x$ and $\alpha,\beta\in\Omega(X,x)$, we may define $f:\bbh\to\bbh$ by $f\circ\ell_1=\alpha$, $f\circ\ell_3=\beta$, and $f\circ\ell_n=c_x$ if $n\notin \{1,3\}$. Any bijection $\phi$ that permutes $1$ and $3$ gives $[\alpha\cdot\beta]=[\prod_{\tau}f\circ\ell_n]=[\prod_{\tau}f\circ\ell_{\phi(n)}]=[\beta\cdot\alpha]$. The argument for the infinite case is similar.
\end{remark}

\begin{theorem}\label{transfinitecharacteriationtheorem}
For any space $X$ and $x\in X$, the following are equivalent:
\begin{enumerate}
\item for every map $f:(\bbh,b_0)\to (X,x)$, there exists a homomorphism $g:\pi_1(\bbt,x_0)\to \pi_1(X,x)$ such that $f_{\#}=g\circ\eta_{\#}$,
\item $X$ is transfinitely $\pi_1$-commutative at $x\in X$,
\item $X$ is infinitely $\pi_1$-commutative at $x\in X$.
\end{enumerate}
\end{theorem}

\begin{proof}
(1) $\Rightarrow$ (2) Suppose that $X$ satisfies (1), $f:(\bbh,b_0)\to (X,x)$ is a map, and $\phi:\bbn\to\bbn$ is a bijection. By assumption, there is a homomorphism $g:\pi_1(\bbt,x_0)\to \pi_1(X,x)$ such that $g\circ \eta_{\#}=f_{\#}$. Set $\alpha_n=f\circ\ell_n$ for $n\in\bbn$ and recall that $\ell_{\tau}=\prod_{\tau}\ell_n$. If we identify $\pi_1(\bbt,x_0)$ with $\bbz^{\bbn}$ in the natural way, then $\eta_{\#}([\ell_{\tau}])=(1,1,1,\dots)$. Since each $\ell_n$ appears exactly once in the concatenation $\prod_{\tau}\ell_{\phi(n)}$, we also have $\eta_{\#}([\prod_{\tau}\ell_{\phi(n)}])=(1,1,1,\dots)$. Therefore, $[\prod_{\tau}\alpha_n]=f_{\#}([\prod_{\tau}\ell_n])=g(\eta_{\#}([\prod_{\tau}\ell_n]))=g(1,1,1,\dots)$ and similarly, we have $[\prod_{\tau}\alpha_{\phi(n)}]=g(1,1,1,\dots)$. Thus $[\prod_{\tau}\alpha_n]=[\prod_{\tau}\alpha_{\phi(n)}]$, proving transfinite $\pi_1$-commutativity.

(2) $\Rightarrow$ (3) Suppose $X$ is transfinitely $\pi_1$-commutative at $x\in X$, $\alpha_n\in\Omega(X,x)$ is a null-sequence, and $\phi:\bbn\to\bbn$ is a bijection. Define maps $f,g:\bbh\to X$ so that $f\circ \ell_{2^k-1}=\alpha_k$, $g\circ\ell_{2^k-1}=\beta_k$ for all $k\in\bbn$ and $f\circ\ell_n=g\circ\ell_n=c_x$ otherwise. Choose any bijection $\psi:\bbn\to\bbn$ satisfying $\psi(2^k-1)=2^{\phi(k)}-1$ for all $k\in\bbn$. By this choice, we have $\prod_{\tau}f\circ\ell_n\simeq \prod_{k=1}^{\infty}\alpha_k$ and $\prod_{\tau}g\circ\ell_n\simeq \prod_{k=1}^{\infty}\alpha_{\phi(k)}$ by collapsing constant subpaths. Thus $\left[\prod_{k=1}^{\infty}\alpha_k\right]=\left[\prod_{\tau}f\circ\ell_n\right] =\left[\prod_{\tau}f\circ\ell_{\psi(n)}\right]=\left[\prod_{\tau}g\circ\ell_n\right]=\left[\prod_{k=1}^{\infty}\alpha_{\phi(k)}\right]$, where the middle equality follows from transfinite $\pi_1$-commutativity.

(3) $\Rightarrow$ (1) Suppose $X$ is infinitely $\pi_1$-commutative at $x\in X$ and $f:(\bbh,b_0)\to (X,x)$ is a map. To show that such a homomorphism $g$ exists, it suffices to show that $f_{\#}(\ker(\eta_{\#}))=1$. Let $1\neq[\alpha]\in \ker(\eta_{\#})$. By Lemma \ref{infinitecommutatorlemma}, we may assume 
\begin{equation}\label{eqn1}\tag{$\ast$}
\alpha \equiv\prod_{n=1}^{\infty}\left(\prod_{i=1}^{k_n}\alpha_{n,i}\cdot\beta_{n,i}\cdot\alpha_{n,i}^{-}\cdot\beta_{n,i}^{-}\right)
\end{equation} where $\alpha_{n,i},\beta_{n,i}\in \Omega(\bbh_{\geq n},b_0)$. Notice the factors of this concatenation have order type $\omega$, so we may write $\alpha\equiv\prod_{m=1}^{\infty}\delta_m$ where $\delta_m$ are the factors $\alpha_{n,i}$, $\beta_{n,i}$, $\beta_{n,i}^{-}$, or $\beta_{n,i}^{-}$ appearing in the same order as the concatenation in (\ref{eqn1}). Let $\phi:\bbn\to\bbn$ be the bijection defined so that for all $j\in\bbn$, we have $\phi(4j-3)=4j-3$, $\phi(4j-2)=4j-1$, $\phi(4j-1)=4j-2$, and $\phi(4j)=4j$. By infinite $\pi_1$-commutativity, we have $[f\circ\alpha]=\left[\prod_{m=1}^{\infty}f\circ\delta_m\right]=\left[\prod_{m=1}^{\infty}f\circ\delta_{\phi(m)}\right]$. The bijection $\phi$ was defined so that 
\begin{equation*}\label{eqn2}
\prod_{m=1}^{\infty}\delta_{\phi(m)} \equiv\prod_{n=1}^{\infty}\left(\prod_{i=1}^{k_n}\alpha_{n,i}\cdot\alpha_{n,i}^{-}\cdot\beta_{n,i}\cdot\beta_{n,i}^{-}\right).
\end{equation*}
The concatenation on the right is a reparemeterization of an infinite concatenation of consecutive inverse pairs and therefore is null-homotopic in $\bbh$. Thus, since $\left[\prod_{m=1}^{\infty}\delta_{\phi(m)}\right]=1$ in $\pioneh$, we have $f_{\#}([\alpha])= f_{\#}\left(\left[\prod_{m=1}^{\infty}\delta_{\phi(m)}\right]\right)=1$ in $\pi_1(X,x)$.
\end{proof}

In Condition (1) of Theorem \ref{transfinitecharacteriationtheorem}, note that if such a $g$ exists, it is necessarily unique by the surjectivity of $\eta_{\#}$. For the remainder of this subsection, we consider some consequences of Theorem \ref{transfinitecharacteriationtheorem}. Applying, once again, the fact that $\eta_{\#}([\ell_{\tau}])=\eta_{\#}([\ell_{\infty}])$, the next result follows immediately.

\begin{corollary}\label{basiccorollary}
If $X$ has transfinite $\pi_1$-commutativity at $x\in X$, then for every null-sequence $\alpha_n\in\Omega(X,x)$, we have $\left[\prod_{\tau}\alpha_n\right]=\left[\prod_{n=1}^{\infty}\alpha_n\right]$.
\end{corollary}

For a compact nowhere dense subset $A\subseteq \ui$, let $\mci(A)$ denote the set of connected components of $[\min(A),\max(A)]\backslash A$ with the natural linear ordering inherited from $\ui$. For instance, the set $\mci(\mcc)$ of components of the complement of the Cantor set $\mcc$ in $I$ has the order type of the rationals.

\begin{lemma}\label{generalpermutationlemma}
Suppose $X$ has transfinite $\pi_1$-commutativity at $x\in X$ and $\alpha,\beta\in\Omega(X,x)$ are non-constant paths. If there is a bijection $\psi:\mci(\alpha^{-1}(x))\to \mci(\beta^{-1}(x))$ such that for every $J\in \mci(\alpha^{-1}(x))$, we have $\alpha|_{\ov{J}}\equiv \beta|_{\ov{\psi(J)}}$, then $[\alpha]=[\beta]$ in $\pi_1(X,x)$.
\end{lemma}

\begin{proof}
If $\mci(\alpha^{-1}(x))$ is finite, then the conclusion follows from the fact that $\pi_1(X,x)$ is abelian. Suppose $\mci(\alpha^{-1}(x))$ is infinite. Since $\mci(\mcc)$ is a dense countable order, there exist order embeddings $\mu:\mci(\alpha^{-1}(x))\to \mci(\mcc)$ and $\nu:\mci(\beta^{-1}(x))\to \mci(\mcc)$. Find a bijection $\Psi:\mci(\mcc)\to\mci(\mcc)$, which extends $\psi$ in the sense that $\Psi\circ\mu=\nu\circ\psi$. Define a bijection $\phi:\bbn\to\bbn$ by $\phi=\theta^{-1}\circ\Psi\circ\theta$. Hence, we have the following commutative diagram of linear orders where $\mu$ and $\nu$ are order-preserving and all other morphisms are set-bijections.
\[\xymatrix{
\mci(\alpha^{-1}(x)) \ar[d]_-{\psi} \ar[r]^-{\mu} & \mci(\mcc) \ar[d]_-{\Psi}  & \bbn \ar[d]^-{\phi} \ar[l]_-{\theta} \\
\mci(\beta^{-1}(x)) \ar[r]_-{\nu} & \mci(\mcc) &  \bbn \ar[l]^-{\theta}
}\]
For $n\in\bbn$, let $\gamma_n\equiv \alpha|_{\ov{J}}$ if $J\in \mci(\alpha^{-1}(x))$ and $\theta(n)=\mu(J)$ and let $\gamma_n=c_x$ be constant otherwise. As before, we have $\prod_{\tau}\gamma_n\simeq \alpha$ by collapsing constant loops. Similarly, let $\delta_m\equiv \beta|_{\ov{K}}$ if $K\in \mci(\beta^{-1}(x))$ and $\theta(m)=\nu(K)$ and let $\delta_m=c_x$ be constant otherwise. From this choice, we have $\prod_{\tau}\delta_m\simeq \beta$ by collapsing constant loops. 

Fix $n\in\bbn$. If $\theta(n)=\mu(J)$, then we have $\theta(\phi(n))=\nu(\psi(J))$ and thus $\gamma_n\equiv \alpha|_{\ov{J}}\equiv \beta|_{\ov{\psi(J)}}\equiv \delta_{\phi(n)}$. If $\theta(n)\notin Im(\mu)$, then $\theta(\phi(n))\notin Im(\mu)$ and we have $\gamma_n=\delta_{\phi(n)}=c_x$. Hence, $\gamma_n\equiv \delta_{\phi(n)}$ for all $n\in\bbn$. This gives
\[[\alpha]=\left[\prod_{\tau}\gamma_n\right]= \left[\prod_{\tau}\delta_{\phi(n)}\right]=\left[\prod_{\tau}\delta_{n}\right]=[\beta]\]
where the second to last equality is given by transfinite $\pi_1$-commutativity.
\end{proof}

As a specific example, and analogue of infinite double series, the next corollary applies Lemma \ref{generalpermutationlemma} to infinite concatenations of order type $\omega^2$.

\begin{corollary}
If $X$ has transfinite $\pi_1$-commutativity at $x\in X$, and $\alpha_{m,n}\in \Omega(X,x)$, $m,n\in\bbn$ is a doubly indexed sequence such that every neighborhood of $e$ contains all but finitely many of the images $Im(\alpha_{m,n})$, then, in $\pi_1(X,x)$, we have:\[\left[\prod_{m=1}^{\infty}\prod_{n=1}^{\infty}\alpha_{m,n}\right]=\left[\prod_{n=1}^{\infty}\prod_{m=1}^{\infty}\alpha_{m,n}\right].\]
\end{corollary}

It will also be useful to know that the property of being transfinitely $\pi_1$-commutative at a given basepoint is invariant under basepoint-preserving homotopy equivalence.

\begin{lemma}\label{homequivlemma}
Suppose $k:X\to Y$ is a homotopy equivalence. If $x\in X$ and $Y$ is transfinitely $\pi_1$-commutative at $k(x)$, then $X$ is transfinitely $\pi_1$-commutative at $x$. In particular, if $f:(X,x)\to (Y,y)$ is a based homotopy equivalence, then $X$ is transfinitely $\pi_1$-commutative at $x$ if and only if $Y$ is transfinitely $\pi_1$-commutative at $y$.
\end{lemma}

\begin{proof}
We apply Theorem \ref{transfinitecharacteriationtheorem} in both directions. If $f:(\bbh,b_0)\to (X,x)$ is a map, then since $Y$ is is transfinitely $\pi_1$-commutative at $k(x)$, there exists a unique homomorphism $g':\bbz^{\bbn}\to \pi_1(Y,k(x))$ such that $g'\circ\eta_{\#}=(k\circ f)_{\#}$. Now $g=k_{\#}^{-1}\circ g'$ is a unique homomorphism $\bbz^{\bbn}\to \pi_1(X,x)$ satisfying $g\circ\eta_{\#}=f_{\#}$. Thus $X$ is transfinitely $\pi_1$-commutative at $x$. The second statement of the lemma follows easily by applying the first statement to a based homotopy equivalence and its based homotopy inverse.
\end{proof}

We identify an algebraic condition that implies transfinite $\pi_1$-commutativity independent of basepoint. An abelian group $A$ is \textit{cotorsion} provided that whenever $A\leq G$ with $G$ abelian and $ G/A$ torsion-free, we have $G= A \oplus B$ for some $B\leq G$. The abelian group $A$ is called \textit{cotorsion-free} if it does not contain a non-trivial cotorsion subgroup. It is shown in \cite{EKH1ofHE} that the kernel of the induced homomorphism $\eta_{\ast}:H_1(\bbh)\to H_1(\bbt)$ is cotorsion and hence $H_1(\bbh)\cong H_1(\bbt)\oplus \ker(\eta_{\ast})$. Let $h_X:\pi_1(X,x)\to H_1(X)$ denote the Hurewicz homomorphism for a space $X$.

\begin{lemma}\label{cotorsionfreetheorem}
If $\pi_1(X,x)$ is a cotorsion-free abelian group, then $X$ is transfinitely $\pi_1$-commutative at all of its points.
\end{lemma}

\begin{proof}
If $\pi_1(X,x)$ is cotorsion-free, then $\pi_1(X,y)$ is cotorsion-free for all $y\in X$. Hence, it suffices to show that $X$ is transfinitely $\pi_1$-commutative at $x\in X$. Let $f:(\bbh,b_0)\to (X,x)$ be a map. Since $\pi_1(X,x)$ is abelian, there is a unique homomorphism $g':H_1(\bbh)\to \pi_1(X,x)$ such that $g'\circ h_{\bbh}=f_{\#}$. Since $\ker(\eta_{\ast})$ is cotorsion and any homomorphic image of a cotorsion group is cotorsion, we must have $g'(\ker(\eta_{\ast}))=0$. Hence, there is a unique homomorphism $g'':H_1(\bbh)/\ker(\eta_{\ast})\to \pi_1(X,x)$ such that if $p:H_1(\bbh)\to H_1(\bbh)/\ker(\eta_{\ast})$ is the projection, then $g''\circ p=g'$. Let $k:H_1(\bbt)\to H_1(\bbh)/\ker(\eta_{\ast})$ be the canonical isomorphism  and set $g=g''\circ k\circ h_{\bbt}$. From the naturality of the Hurewicz homomorphisms it follows directly that $g\circ \eta_{\#}=f_{\#}$, verifying (1) in Theorem \ref{transfinitecharacteriationtheorem}.
\end{proof}

\begin{example}
One may construct a CW-complex $X$, with fundamental group isomorphic to $\bbz^{\bbn}/\oplus_{\bbn}\bbz$, which is not a cotorsion-free group. However, since $X$ is locally contractible, $X$ is (trivially) transfinitely $\pi_1$-commutative at all of its points. Hence, the converse of Lemma \ref{cotorsionfreetheorem} is far from being true. This example emphasizes the fact that the transfinitely $\pi_1$-commutative property is an invariant property of the infinitary structure of fundamental groups inherited from the loop space and not the underlying (finitary) group structure.
\end{example}

\begin{example}
It is also worth noting that the transfinitely $\pi_1$-commutative property is not a purely local property since it is possible for a space to fail to be transfinitely $\pi_1$-commutative at a point due to the global structure of the space. For example, if $X=\bbh\cup I\cup \bbt/\mathord{\sim}$ where $b_0\sim 0$ and $x_0\sim 1$, then $\pi_1(X,x_0)$ is isomorphic to the free product $\pi_1(\bbh,b_0)\ast \pi_1(\bbt,x_0)$. Since $\pi_1(X,x_0)$ is not abelian, $X$ is not transfinitely $\pi_1$-commutative at any point. However, the open set $U\subseteq X$, which is the image of $(1/2,1]\cup \bbt$ in $X$, has cotorsion-free fundamental group isomorphic to $\bbz^{\bbn}$. Hence, by Theorem \ref{cotorsionfreetheorem}, $U$ is transfinitely $\pi_1$-commutative at all of its points. Taking this argument further, one could show that $X$ is ``locally" transfinitely $\pi_1$-commutative at $x_0$. Indeed, there are a number of local/semilocal variants of the transfinitely $\pi_1$-commutative property that one could compare and contrast.
\end{example}

\subsection{Well-definedness of operations on homotopy classes}

We have chosen to avoid infinite sum notation $\sum_{n=1}^{\infty}[\alpha_n]$ to represent the homotopy class of an infinite concatenation $\prod_{n=1}^{\infty}\alpha_n$ even in the presence of transfintite $\pi_1$-commutativity because the infinitary operation $\{\alpha_n\}\mapsto \left[\prod_{n=1}^{\infty}\alpha_n\right]$ on null-sequences does not always induce a well-defined operation $\{[\alpha_n]\}\mapsto \left[\prod_{n=1}^{\infty}\alpha_n\right]$ on homotopy classes. The well-definedness of such infinitary operations on fundamental groups has been studied in significant detail in \cite{Brazscattered,BFTestMap}.

\begin{definition}
We say that a space $X$
\begin{enumerate}
\item has \textit{well-defined infinite $\pi_1$-products at $x\in X$} if for any null-sequences $\alpha_n,\beta_n\in\Omega(X,x)$ such that $\alpha_n\simeq\beta_n$ for all $n\in\bbn$, we have $\prod_{n=1}^{\infty}\alpha_n\simeq \prod_{n=1}^{\infty}\beta_n$,
\item has \textit{well-defined transfinite $\pi_1$-products at $x\in X$} if for any null-sequences $\alpha_n,\beta_n\in\Omega(X,x)$ such that $\alpha_n\simeq\beta_n$ for all $n\in\bbn$, we have $\prod_{\tau}\alpha_n\simeq \prod_{\tau}\beta_n$,
\item is \textit{homotopically Hausdorff} at $x\in X$ if for every $[\alpha]\in\pi_1(X,x)$, there exists a neighborhood $U$ of $x$ such that $\Omega(U,x)\cap [\alpha]=\emptyset$.
\end{enumerate}
\end{definition}

\begin{remark} \label{welldefinedremark}
In general, we have (3) $\Rightarrow$ (1) and (2) $\Rightarrow$ (1) where (1) $\Leftrightarrow$ (3) if $X$ is first countable at $x$; see \cite[Section 3]{BFTestMap}. It is an open problem if (1) $\Leftrightarrow$ (2) holds in general. We refer to \cite{Brazscattered,BFTestMap,CMRZZ08,FRVZ11,FZ07} for more on the homotopically Hausdorff property and its relationship to other properties of fundamental group(oid)s.
\end{remark}


\begin{lemma}\label{wdinfiniteimplieswdtransfinitelemma}
Suppose $X$ has transfinite $\pi_1$-commutativity at $x\in X$. Then $X$ has well-defined infinite $\pi_1$-products at $x$ if and only if $X$ has well-defined transfinite $\pi_1$-products at $x$. Moreover, these are equivalent to being homotopically Hausdorff at $x$ if $X$ is first countable at $x$.
\end{lemma}

\begin{proof}
In light of Remark \ref{welldefinedremark}, it suffices to show that, assuming the hypothesis, well-defined infinite $\pi_1$-products at $x$ $\Rightarrow$ well-defined transfinite $\pi_1$-products at $x$. Suppose $X$ has well-defined infinite $\pi_1$-products at $x$ and $\alpha_n,\beta_n\in\Omega(X,x)$ are null-sequences such that $\alpha_n\simeq \beta_n$ for all $n\in\bbn$. Consider the maps $f_1,f_2:\bbh\to X$ with $f_1\circ\ell_n=\alpha_n$ and $f_2\circ\ell_n=\beta_n$. By assumption, we have $(f_1)_{\#}([\ell_{\infty}])=(f_2)_{\#}([\ell_{\infty}])$. Corollary \ref{basiccorollary} gives $(f_i)_{\#}[\ell_{\infty}]=(f_i)_{\#}([\ell_{\tau}])$ for $i\in\{1,2\}$. Hence, we have $\left[\prod_{\tau}\alpha_n\right]=(f_1)_{\#}([\ell_{\tau}])=(f_2)_{\#}([\ell_{\tau}])=\left[\prod_{\tau}\beta_n\right]$.
\end{proof}

In the next two results, we give each fundamental group the natural quotient topology inherited from the loop space. It is known that this topology on $\pi_1$ is functorial, i.e. a map $f:(X,x)\to (Y,y)$ induces a continuous homomorphism $f_{\#}:\pi_1(X,x)\to \pi_1(Y,y)$. In particular, the quotient topology on $\pi_1(\bbt,x_0)=\bbz^{\bbn}$ agrees with the direct product topology where each factor $\bbz$ is discrete. We refer to \cite{BFqtop} for more on the basic theory of the quotient topology on $\pi_1$.

\begin{lemma}\label{sequenceliftlemma}
If $\{s_k\}_{k\in\bbn}\to s$ is a convergent sequence in $\pi_1(\bbt,x_0)=\bbz^{\bbn}$, then there exists a convergent sequence $\{\beta_k\}_{k\in\mathbb{N}}\to \beta$ in $\Omega(\bbh,b_0)$ such that $\eta_{\#}([\beta_k])=s_k$ and $\eta_{\#}([\beta])=s$. In particular, $\eta_{\#}$ is a topological quotient map.
\end{lemma}
\begin{proof}
Write $s_k=(a_{1}^{k},a_{2}^{k},a_{3}^{k},\dots)$ and $s=(b_1,b_2,b_3,\dots)$ as sequences of integers. Since $\bbz^{\bbn}$ has the product topology, if $\{s_k\}\to s$, then for every $n\in\bbn$, there exists $K_n\in\bbn$ such that $a_{n}^{k}=b_n$ for all $k\geq K_n$. Consider the infinite concatenations in $\bbh$: $\beta_{k}=\prod_{n=1}^{\infty}\ell_{n}^{a_{n}^{k}}$ and $\beta=\prod_{n=1}^{\infty}\ell_{n}^{b_n}$ where $\ell_{n}^{0}$ denotes $c_{b_0}$. By construction, we have $\eta_{\#}([\beta_k])=s_k$ and $\eta_{\#}([\beta])=s$. Since the $n$-th factor in the sequence $\beta_{k}$ stabilizes to the $n$-th factor of $\beta$, it follows that $\{\beta_k\}\to \beta$ in the compact-open topology on $\Omega(\bbh,b_0)$. 

The identifying map $\pi:\Omega(\bbh,b_0)\to \pioneh$ is assumed to be continuous (and quotient). Note that $\bbz^{\bbn}$ is a sequential space and, by the previous paragraph, $\eta_{\#}\circ\pi:\Omega(\bbh,b_0)\to\pi_1(\bbt,x_0)$ lifts all convergent sequences. It follows that $\eta_{\#}\circ\pi$ is a quotient map and thus $\eta_{\#}$ is quotient as well.
\end{proof}

Recalling the notion of a ``cotorsion-free" abelian group from the previous section, we note that it is well-known from infinite abelian group theory that an abelian group $A$ is cotorsion-free if and only if for every homomorphism $g:\bbz^{\bbn}\to A$, we have $\bigcap_{n\in\bbn}g\left(\prod_{k=n}^{\infty}\bbz\right)=0$ (See \cite[Theorem 7.2]{CC}). We consider a topological analogue of this property in the next corollary where $\ov{0}$ will denote the closure of the trivial subgroup in the group $\pi_1(X,x)$ equipped with the quotient topology.

\begin{corollary}\label{wdtransfiniteprodcorollary}
Suppose $X$ has transfinite $\pi_1$-commutativity at $x\in X$. If for every continuous homomorphism $g:\bbz^{\bbn}\to \ov{0}$ to the closure of the trivial subgroup in $\pi_1(X,x)$, we have $\bigcap_{n\in\bbn}g\left(\prod_{k=n}^{\infty}\bbz\right)=0$, then $X$ has well-defined transfinite $\pi_1$-products at $x$.
\end{corollary}
\begin{proof}
In light of Remark \ref{welldefinedremark} and Lemma \ref{wdinfiniteimplieswdtransfinitelemma}, it suffices to show that $X$ has well-defined infinite $\pi_1$-products at $x$. Suppose to the contrary that $X$ is transfinitely $\pi_1$-commutative at $x$ but does not have well-defined infinite $\pi_1$-products at $x$, then by \cite[Theorem 5.12]{Brazscattered}, there exists a map $f:(\bbh,b_0)\to (X,x)$ such that $f_{\#}([\ell_n])=0$ for all $n\in\bbn$ and $f_{\#}([\ell_{\infty}])\neq 0$. By Theorem \ref{welldefinedremark}, there exists a homomorphism $g:\bbz^{\bbn}\to \pi_1(X,x)$ such that $g\circ\eta_{\#}=f_{\#}$. Since $f_{\#}$ is continuous by the functorality of the quotient topology and $\eta_{\#}$ is quotient by Lemma \ref{sequenceliftlemma}, it follows that $g$ is continuous. Since $g(\mathbf{e}_n)=g(\eta_{\#}([\ell_n]))=f_{\#}([\ell_n])=0$ for all $n\in\bbn$, we have $g\left(\oplus_{\bbn}\bbz\right)=0$. The continuity of $g$ now gives $g(\bbz^{\bbn})=g\left(\ov{\oplus_{\bbn}\bbz}\right)\leq \ov{g(\oplus_{\bbn}\bbz)}=\ov{0}$. Hence, we may regard $g$ as a continuous homomorphism $g:\bbz^{\bbn}\to\ov{0}$.

Let $\mathbf{a}_1=(1,1,1,\dots)$ and, for $n\geq 2$, let $\mathbf{a}_n=\mathbf{a}_1-\sum_{i=1}^{n-1}\mathbf{e}_i\in\prod_{k=n}^{\infty}\bbz$ be the sequence with $0$'s in the first $n$ components and $1$'s in every remaining component. Since $g(\mathbf{a}_1)=g(\eta_{\#}([\ell_{\infty}]))=f_{\#}([\ell_{\infty}])\neq 0$ and $g\left(\oplus_{\bbn}\bbz\right)=0$, we have $g(\mathbf{a}_1)=g(\mathbf{a}_n)$ for all $n\in\bbn$. Hence $0\neq g(\mathbf{a}_1)\in \bigcap_{n\in\bbn}g\left(\prod_{k=n}^{\infty}\bbz\right)$.
\end{proof}

\begin{remark}
Many topologies on $\pi_1$ have been considered in the literature. We note that Corollary \ref{wdtransfiniteprodcorollary} holds if one replaces the quotient topology with any functorial topology on $\pi_1$ for which (1) $\pi_1$ becomes at least a quasitopological group (so that $\ov{0}$ is a subgroup), (2) $\pi_1(\bbt,x_0)\cong\bbz^{\bbn}$ has the product topology (so that $\oplus_{\bbn}\bbz$ is dense), and (3) $\eta_{\#}$ is a topological quotient map.
\end{remark}

Combining Lemma \ref{cotorsionfreetheorem} with Corollary \ref{wdtransfiniteprodcorollary} gives the following.

\begin{theorem}\label{cotorsionfreemainthm}
If $\pi_1(X,x)$ is a cotorsion-free abelian group, then $X$ is transfinitely $\pi_1$-commutative and has well-defined transfinite $\pi_1$-products at all of its points.
\end{theorem}

\section{Fundamental groups of monoids}\label{sectionmonoids}

Throughout this section, $M$ will represent a given path-connected pre-$\Delta$-monoid $M$ with identity $e\in M$. To show that $M$ is transfinitely $\pi_1$-commutative at $e$, we begin, in Section \ref{subsectionelementarymonoids}, by expanding upon some familiar results on fundamental groups of monoids. In particular, we will focus on controlling of the size of the homotopies used, making a point to bound the image of a homotopy whenever possible. The strictness of both identity and associativity in $M$ will be used implicitly and this appears to be crucial to the argument. Since all of the properties we consider are invariant under basepoint-preserving homotopy equivalence, our arguments also apply to any based space $(X,x)$ that is based homotopy equivalent to a pre-$\Delta$ monoid $(M,e)$, e.g. loop spaces. Such spaces might be considered ``strong $H$-spaces" and include nearly all $H$-spaces of interest.

\subsection{Elementary homotopies and their images}\label{subsectionelementarymonoids}

We recall the usual \textit{distributive law} for paths: given paths $\alpha_{i,j}:I\to M$, $1\leq i,j\leq n$ satisfying $\alpha_{i+1,j}(0)=\alpha_{i,j}(1)$, we have $\bigast_{j=1}^{n}\prod_{i=1}^{n}\alpha_{i,j}=\prod_{i=1}^{n}\bigast_{j=1}^{n}\alpha_{i,j}$, which is a path from $\bigast_{j=1}^{n}\alpha_{1,j}(0)$ to $\bigast_{j=1}^{n}\alpha_{n,j}(1)$ in $M$.



\begin{lemma}\label{product-to-concatentation}
Let $\alpha_1,\alpha_2,\dots,\alpha_n:\ui\to M$ be paths and $\phi:\{1,2,\dots,n\}\to\{1,2,\dots,n\}$ be any bijection. For $j,k\in \{1,2,\dots,n\}$, let $$\beta_{j,k}:=\begin{cases}
\alpha_k(0), & \text{ if }\phi(k)\in \{j+1,j+2,\dots,n\}\\
\alpha_k, & \text{ if }\phi(k)=j\\
\alpha_k(1), & \text{ if }\phi(k)\in \{1,2,\dots,j-1\}.
\end{cases}$$
Then $\bigast_{k=1}^{n} \alpha_k\simeq \prod_{j=1}^{n}\left(\bigast_{k=1}^{n}\beta_{j,k}\right)$ by a homotopy with image in $\bigast_{k=1}^{n}Im(\alpha_k)$.
\end{lemma}

\begin{proof}
Replace $\alpha_{k}(0)$ and $\alpha_k(1)$ in the definition of $\beta_{j,k}$ with the constant paths $c_{\alpha_k(0)}$ and $c_{\alpha_k(1)}$ respectively. The value of $\prod_{j=1}^{n}\left(\bigast_{k=1}^{n}\beta_{j,k}\right)$ is the same as in the statement of the lemma and the distributive law gives $\prod_{j=1}^{n}\left(\bigast_{k=1}^{n}\beta_{j,k}\right)=\bigast_{k=1}^{n}\left(\prod_{j=1}^{n}\beta_{j,k}\right)$. For each $k\in\{1,2,\dots,n\}$, every factor $\beta_{j,k}$ of the concatenation $\prod_{j=1}^{n}\beta_{j,k}$ is constant except for $\beta_{\phi(k),k}=\alpha_k$. Hence, there is a path-homotopy $H_k:I^2\to M$ from $\prod_{j=1}^{n}\beta_{j,k}$ to $\alpha_k$ with image in $Im(\alpha_k)$. Since $M$ is a pre-$\Delta$-monoid and $I^2$ is $\Delta$-generated, the monoid product $\bigast_{k=1}^{n}H_k:I^2\to M$ is continuous and gives a path-homotopy from $\bigast_{k=1}^{n}\prod_{j=1}^{n}\beta_{j,k}$ to $\bigast_{k=1}^{n}\alpha_k$ with image in $\bigast_{k=1}^{n}Im(\alpha_k)$.
\end{proof}

\begin{example}
In the case that $n=2$, Lemma \ref{product-to-concatentation} states that for any paths $\alpha,\beta:I\to M$, the three paths $\alpha\ast\beta$, $(\alpha\ast\beta(0))\cdot (\alpha(1)\ast\beta)$, and $(\alpha(0)\ast\beta)\cdot(\alpha\ast \beta(1))$, are all homotopic to each other by homotopies in $Im(\alpha)\ast Im(\beta)$. For a more complicated example, consider paths $\alpha_k:\ui\to M$, $1\leq k\leq 4$ and the bijection $\phi$ given by $\phi(1)=2$, $\phi(2)=3$, $\phi(3)=1$, $\phi(4)=4$. Then the product $\bigast_{k=1}^{4} \alpha_k$ is homotopic (relative to the endpoints $\bigast_{k=1}^{4} \alpha_k(0)$ and $\bigast_{k=1}^{4} \alpha_k(1)$) to the concatenation 
\begin{eqnarray*}
\left(\alpha_1(0)\ast\alpha_2(0)\ast \alpha_3\ast \alpha_{4}(0)\right)\cdot\left(\alpha_1\ast\alpha_2(0)\ast \alpha_3(1)\ast \alpha_{4}(0)\right)\cdot \\
\left(\alpha_1(1)\ast\alpha_2\ast \alpha_3(1)\ast \alpha_{4}(0)\right)\cdot\left(\alpha_1(1)\ast\alpha_2(1)\ast \alpha_3(1)\ast \alpha_{4}\right)
\end{eqnarray*}
by a homotopy in $\bigast_{k=1}^{4}Im(\alpha_k)$.
\end{example}


\begin{corollary}\label{smallcommutinghomotopylemma}
For all $\alpha,\beta\in \Omega(M,e)$, the loops $\alpha\ast\beta$, $\alpha\cdot \beta$, and $\beta\cdot\alpha$ are all homotopic to each other by homotopies in $Im(\alpha)\ast Im(\beta)$. Moreover, $\beta\ast \alpha$ is homotopic to each of these by a homotopy in $(Im(\alpha)\ast Im(\beta))\cup (Im(\beta)\ast Im(\alpha))$.
\end{corollary}

\begin{proof}
Since $e$ is a strict identity for the operation of $M$, we have $\alpha\ast\beta(0)=\alpha=\alpha\ast\beta(1)$
and $\alpha(1)\ast\beta=\beta=\alpha(0)\ast\beta$. Applying Lemma \ref{product-to-concatentation} in the case $n=2$ gives that $\alpha\ast\beta$, $\alpha\cdot\beta$, and $\beta\cdot \alpha$ are all homotopic to each other by homotopies in $Im(\alpha)\ast Im(\beta)$. By switching the roles of $\alpha$ and $\beta$, we also have that $\beta\ast\alpha$ is homotopic to $\beta\cdot\alpha$ by a homotopy in $Im(\beta)\ast Im(\alpha)$.
\end{proof}
\begin{corollary}
If $M$ is a pre-$\Delta$-monoid, then $\pi_1(M,e)$ is commutative and $[\alpha\cdot\beta]=[\alpha\ast\beta]$ for all $\alpha,\beta\in\Omega(M,e)$.
\end{corollary}

\begin{remark}[Generalized Distributive Law]
For a pre-$\Delta$-monoid $M$, maps $f,g:\bbh\to M$, and loop $L\in \Omega(\bbh,b_0)$, we have \[(f\ast g)\circ L=(f\circ L)\ast (g\circ L)\]
where both sides of the equality are well-defined since $I$ and $\bbh$ are $\Delta$-generated. For instance, setting $L=\ell_{\infty}$, $f\circ\ell_n=\alpha_n$, and $g\circ \ell_n=\beta_n$, we obtain the distributive law for infinite concatenations: $\prod_{n=1}^{\infty}(\alpha_n\ast\beta_n)=\left(\prod_{n=1}^{\infty}\alpha_n\right)\ast \left(\prod_{n=1}^{\infty}\beta_n\right)$.
\end{remark}

\begin{lemma}\label{smallimagelemma}
If $\{\alpha_{n}^{i}\}_{n\in\bbn}$, $1\leq i\leq k$ is a finite collection of null-sequences in $\Omega(M,e)$, then for any open neighborhood $U$ of $e$ in $M$, we have $\bigast_{i=1}^{k}Im(\alpha_{n}^{i})\subseteq U$ for all but finitely many $n$. In particular, $\left\{\bigast_{i=1}^{k}\alpha_{n}^{i}\right\}_{n\in\bbn}$ is a null sequence.
\end{lemma}

\begin{proof}
For $1\leq i\leq k$, let $f_i:\bbh\to M$ be the maps such that $f_i\circ\ell_n=\alpha_{n}^{i}$ for all $n\in\bbn$. Since $M$ is a pre-$\Delta$-monoid, the $k$-ary operation $\mu_k:\Delta(M^k)\to M$ is continuous (recall Proposition \ref{predeltamoncharprop}). Since $\bbh^k$ is a $\Delta$-generated space, $f=\prod_{i=1}^{k}f_i:\bbh^k\to \Delta(M^k)$ is continuous. We have $\mu_k(f(b_0,b_0,\dots,b_0))=e\in U$ and so there exists an open neighborhood $V$ of $b_0$ such that $\mu_k(f(V^k))\subseteq U$. Find $N\in\bbn$ such that $C_n\subseteq V$ for all $n\geq N$. Then $\bigast_{i=1}^{k}Im(\alpha_{n}^{i})=\mu_k(f((C_{n})^{k}))\subseteq \mu_k(f(V^k))\subseteq U$ for all $n\geq N$, completing the proof of the first statement. Since $Im(\bigast_{i=1}^{k}\alpha_{n}^{i})\subseteq \bigast_{i=1}^{k}Im(\alpha_{n}^{i})$ for all $n\in\bbn$, the second statement is an immediate consequence of the first statement.
\end{proof}

\begin{proposition}\label{infprodprop2}
If $\alpha_n,\beta_n\in \Omega(M,e)$ are null sequences, then $\prod_{n=1}^{\infty}(\alpha_n\cdot\beta_n)\simeq\prod_{n=1}^{\infty}(\alpha_n\ast\beta_n)$ by a homotopy in $\bigcup_{n\in\bbn}Im(\alpha_n)\ast Im(\beta_n)$.
\end{proposition}

\begin{proof}
By Lemma \ref{smallcommutinghomotopylemma}, there are homotopies $H_n:I^2\to M$ with $H_n(s,0)=\alpha_n\cdot \beta_n$ and $H_n(s,1)=\alpha_n\ast \beta_n$ such that $Im(H_n)\subseteq Im(\alpha_n)\ast Im(\beta_n)$. Define a homotopy $H$ as $H_n$ on $\left[\frac{n-1}{n},\frac{n}{n+1}\right]\times \ui$ for each $n\in\bbn$ and $H(\{1\}\times\ui)=e$. It suffices to show that $H$ is continuous at the points in $\{1\}\times \ui$. Let $U$ be an open neighborhood of $e$ in $M$. By Lemma \ref{smallimagelemma}, $Im(\alpha_n)\ast Im(\beta_n)$ lies in $U$ for all but finitely many $n$. Therefore, since $Im(H_n)\subseteq Im(\alpha_n)\ast Im(\beta_n)$, we have $H(\left[\frac{N-1}{N},1\right]\times \ui)\subseteq U$ for some $N\in\bbn$, verifying the continuity of $H$. Thus, $H$ is the desired homotopy.
\end{proof}

The following theorem, which says that a concatenation of order type $\omega$ may be reordered into a concatenation of order type $\omega+\omega$ is a first step toward infinite commutativity.

\begin{theorem}\label{infiniteproductscommutesthm}
If $M$ is a pre-$\Delta$-monoid and $\alpha_n,\beta_n\in \Omega(M,e)$ are null sequences of loops, then 
$\prod_{n=1}^{\infty}\alpha_n\cdot\beta_n\simeq\left(\prod_{n=1}^{\infty}\alpha_n\right)\cdot \left(\prod_{n=1}^{\infty}\beta_n\right)$ by a homotopy in $Im(\prod_{n=1}^{\infty}\alpha_n)\ast Im(\prod_{n=1}^{\infty}\beta_n)$.
\end{theorem}

\begin{proof}
We have
\[\prod_{n=1}^{\infty}\alpha_n\cdot\beta_n \simeq \prod_{n=1}^{\infty}\alpha_n\ast\beta_n=\left(\prod_{n=1}^{\infty}\alpha_n\right)\ast \left(\prod_{n=1}^{\infty}\beta_n\right)\simeq \left(\prod_{n=1}^{\infty}\alpha_n\right)\cdot \left(\prod_{n=1}^{\infty}\beta_n\right)\]
where the first homotopy is from Proposition \ref{infprodprop2} and may be chosen to have image in $\bigcup_{n\in\bbn}Im(\alpha_n)\ast Im(\beta_n)$. The equality is the distributive law and the last homotopy, using Corollary \ref{smallcommutinghomotopylemma}, may be chosen to have image in $Im(\prod_{n=1}^{\infty}\alpha_n)\ast Im(\prod_{n=1}^{\infty}\beta_n)$. Since $\bigcup_{n\in\bbn}Im(\alpha_n)\ast Im(\beta_n)\subseteq Im(\prod_{n=1}^{\infty}\alpha_n)\ast Im(\prod_{n=1}^{\infty}\beta_n)$, the composition of these two homotopies has image in $Im(\prod_{n=1}^{\infty}\alpha_n)\ast Im(\prod_{n=1}^{\infty}\beta_n)$.
\end{proof}

Summarizing the last two statements, we conclude that for null sequences $\alpha_n,\beta_n\in\Omega(M,e)$, we have the following equalities in $\pi_1(M,e)$:
\[\left[\prod_{n=1}^{\infty}\alpha_n\right]\left[\prod_{n=1}^{\infty}\beta_n\right]=\left[\prod_{n=1}^{\infty}\alpha_n\ast\beta_n\right] =\left[\prod_{n=1}^{\infty}\alpha_n\cdot\beta_n\right].\]

\subsection{Transfinite $\pi_1$-commutativity for monoids}

In the next lemma, we show that $M$ is infinitely $\pi_1$-commutative at $e$ in a strong way.

\begin{lemma}[Infinite Shuffle]\label{monoidtechlemma}
If $\alpha_n\in\Omega(M,e)$ is a null-sequence and $\phi:\bbn\to\bbn$ is a bijection, then $\prod_{n=1}^{\infty}\alpha_n\simeq \prod_{n=1}^{\infty}\alpha_{\phi(n)}$ by a homotopy in $Im(\prod_{n=1}^{\infty}\alpha_n)\ast Im(\prod_{n=1}^{\infty}\alpha_n)$.
\end{lemma}

\begin{proof}
Let $\alpha_n\in\Omega(M,e)$ be a null sequence and $\phi:\bbn\to\bbn$ be a bijection. Note that that $\{\ell_{\phi(n)}\}$ and $\{\alpha_{\phi(n)}\}$ are null sequences in $\Omega(\bbh,b_0)$ and $\Omega(M,e)$ respectively. For convenience, we reparameterize the infinite concatenations $\ell_{\infty}=\prod_{n=1}^{\infty}\ell_{n}$ and $L=\prod_{n=1}^{\infty}\ell_{\phi(n)}$ in $\bbh$ to be defined respectively as the loops $\ell_{n}$ and $\ell_{\phi(n)}$ on $\left[1-1/2^{n-1},1-1/2^n\right]$ and $\ell_{\infty}(1)=L(1)=b_0$. Find a map $f:\bbh\to M$ for which $\alpha_n=f\circ\ell_n$. We construct a path-homotopy $H:I^2\to M$ from $f\circ L=\prod_{n=1}^{\infty}\alpha_{\phi(n)}$ to $f\circ\ell_{\infty}=\prod_{n=1}^{\infty}\alpha_n$ by recursively defining $H$ piecewise on a sequence of shrinking rectangles in $I^2$.

We being the definition of $H$ as follows: set $H(s,0)=f\circ L(s)$, $H(s,1)=f\circ\ell_{\infty}(s)$, $H(0,t)=H(1,t)=e$, and for $(s,t)\in \left[1-\frac{1}{2^{n-1}},1-\frac{1}{2^n}\right]\times \left[0,\frac{1}{2^n}\right]$, we set $H(s,t)=f\circ L(s)$. We complete the definition of $H$ on the rectangles $R_n=\left[1-\frac{1}{2^{n-1}},1\right]\times\left[\frac{1}{2^{n}},\frac{1}{2^{n-1}}\right]$ by constructing individual path-homotopies $H_n:I^2\to M$ and pre-composing with the canonical affine homeomorphism $R_n\to I^2$.

Set $A_n=\{\phi(1),\dots,\phi(n)\}$ and $B_n=\bbn \backslash A_n$ and consider the subspaces $\bbh_{A_n}$ and $\bbh_{B_n}$ of $\bbh$. The first is a finite wedge of circles and the second is homeomorphic to $\bbh$. Since $\pioneh=\pi_1(\bbh_{A_n},b_0)\ast \pi_1(\bbh_{B_n},b_0)$, there are unique (up to reparameterization) loops $a_1,b_1\in \Omega(\bbh_{B_1},b_0)$ such that $\ell_{\infty}\equiv a_1\cdot\ell_{\phi(1)}\cdot b_1$. Applying the same argument inductively, we may find unique (up to reparameterization) loops $a_n,b_n\in \Omega(\bbh_{B_n},b_0)$ such that $a_{n-1}\cdot b_{n-1}\equiv a_n\cdot\ell_{\phi(n)}\cdot b_n$. Observe that both $\{a_n\}$ and $\{b_n\}$ are null sequences.

Set $\beta_0=f\circ\ell_{\infty}$ and for each $n\in\bbn$, set $\delta_n=f\circ a_n$, $\epsilon_n=f\circ b_n$, and $\beta_n=\delta_n\cdot\epsilon_n$. The continuity of $f$ and the loop-concatenation operation ensures that $\{\delta_n\}$, $\{\epsilon_n\}$, and $\{\beta_n\}$ are all null-sequences in $\Omega(M,e)$.

For every $n\in\bbn$, $\beta_{n-1}$ is a reparameterization of $\delta_n\cdot \alpha_{\phi(n)}\cdot \epsilon_n$. Thus $\beta_{n-1}\simeq \alpha_{\phi(n)}\cdot \delta_n\cdot\epsilon_n\equiv \alpha_{\phi(n)}\cdot \beta_n$. In particular, we may choose a homotopy $\delta_n\cdot \alpha_{\phi(n)}\cdot \epsilon_n\simeq  \alpha_{\phi(n)}\cdot\delta_n\cdot \epsilon_n$, which
\begin{itemize}
\item is the constant homotopy of $\epsilon_n$ on $[2/3,1]\times I$,
\item and commutes $\delta_n$ and $\alpha_{\phi(n)}$ on the domain $[0,2/3]\times I$ by a path-homotopy with image in $Im(\alpha_{\phi(n)})\ast Im(\delta_n)$ (recall Corollary \ref{smallcommutinghomotopylemma}).
\end{itemize}
We conclude that there is a path-homotopy $H_n:\ui^2\to M$ with $H_n(s,0)=(\alpha_{\phi(n)}\cdot \beta_n )(s)$, $H_n(s,1)=\beta_{n-1}(s)$, and such that $Im(H_n)\subseteq (Im(\alpha_{\phi(n)})\ast Im(\delta_n))\cup Im(\epsilon_n)$. This completes the definition of $H$ (See Figure \ref{homotopyfigure}). For all $t\in I$ and $n\in\bbn$, we have $\alpha_{\phi(n)}(t),\delta_n(t),\epsilon_n(t)\in Im\left(\prod_{n=1}^{\infty}\alpha_n\right)$, making it clear that $Im(H)\subseteq Im\left(\prod_{n=1}^{\infty}\alpha_n\right)\ast Im\left(\prod_{n=1}^{\infty}\alpha_n\right)$.

The continuity of $H$ is clear at every point except for $(1,0)\in I^2$. However, all three sequences $\alpha_{\phi(n)}$, $\delta_n$, and $\epsilon_n$ are null-sequences based at $e$. By Lemma \ref{smallimagelemma}, every neighborhood $U$ of $e$ in $M$ contains all but finitely many of the sets $(Im(\alpha_{\phi(n)})\ast Im(\delta_n))\cup Im(\epsilon_n)$. Therefore, all but finitely many of the sets $Im(H_n)\cup Im(\alpha_{\phi(n)})$ lie in $U$. Continuity at $(1,0)$ is immediate from this observation.
\end{proof}

\begin{figure}[t]
\centering \includegraphics[height=3.5in]{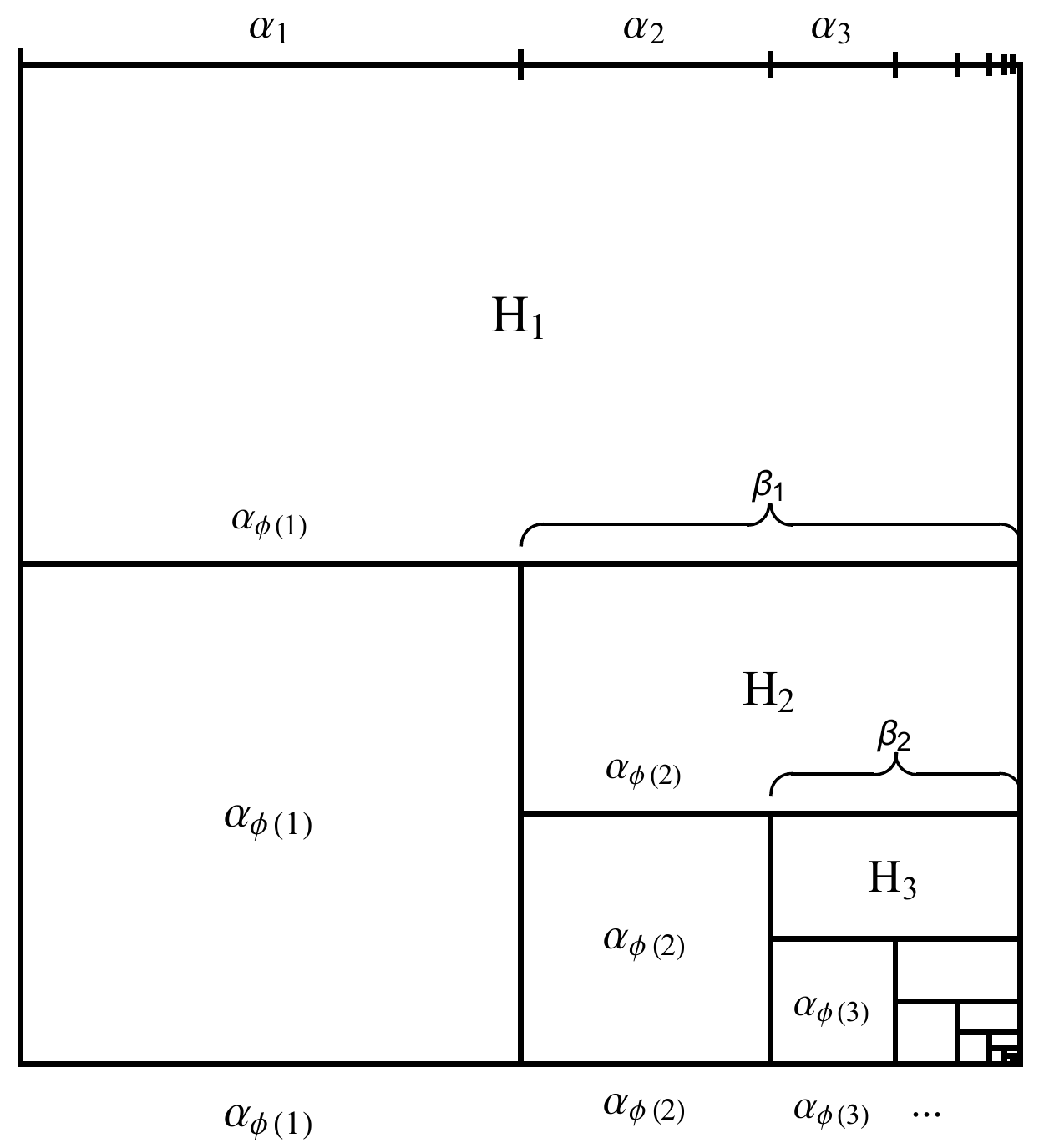}
\caption{\label{homotopyfigure}The homotopy constructed in Lemma \ref{monoidtechlemma}.}
\end{figure}

\begin{theorem} \label{monoidtransfinitecommutethm}
Every pre-$\Delta$-monoid $M$ is transfinitely $\pi_1$-commutative at its identity element.
\end{theorem}

\begin{proof}
By Lemma \ref{monoidtechlemma}, $M$ is infinitely $\pi_1$-commutative at the identity $e\in M$. By Theorem \ref{transfinitecharacteriationtheorem}, $M$ is also transfinitely $\pi_1$-commutative at $e$.
\end{proof}

\begin{remark}
The homotopy $H$ in Lemma \ref{monoidtechlemma} is obtained by gluing together constant homotopies and the homotopies $H_n$, which are themselves an Eckmann-Hilton-type shuffle followed by a constant homotopy. It is in this sense that we consider this infinite shuffle to be an analogue of the Eckmann-Hilton Principle.

We also note that one may replace $\prod_{n=1}^{\infty}\alpha_n$ with $\prod_{\tau}\alpha_n$ in Lemma \ref{monoidtechlemma} and the exact same argument will show that for any bijection $\phi:\bbn\to\bbn$, we have $\prod_{\tau}\alpha_n\simeq \prod_{n=1}^{\infty}\alpha_{\phi(n)}$ by a homotopy in $A=Im(\prod_{\tau}\alpha_n)\ast Im(\prod_{\tau}\alpha_n)=Im(\prod_{n=1}^{\infty}\alpha_n)\ast Im(\prod_{n=1}^{\infty}\alpha_n)$. Using $\phi^{-1}$ gives $\prod_{\tau}\alpha_{\phi(n)}\simeq \prod_{n=1}^{\infty}\alpha_{n}$ in $A$ and concatenating these three homotopies provides a homotopy $\prod_{\tau}\alpha_n\simeq \prod_{\tau}\alpha_{\phi(n)}$ with image in $A$. Hence, one may verify that $M$ is transfinitely $\pi_1$-commutative at $e\in M$ directly (without appealing to Theorem \ref{transfinitecharacteriationtheorem}) while maintaining control of the size of the homotopy. 
\end{remark}

Combining  Theorems \ref{transfinitecharacteriationtheorem} and \ref{monoidtransfinitecommutethm} gives the following result, which will be used in Section \ref{sectionjrproduct}.

\begin{corollary}\label{factorizationlemma}
If $M$ is a pre-$\Delta$-monoid with identity $e$, then for every map $f:(\bbh,b_0)\to (M,e)$ there exists a unique homomorphism $g:\bbz^{\bbn}\to \pi_1(M,e)$ such that $g\circ\eta_{\#}=f_{\#}$.
\end{corollary}

\begin{corollary}
Let $M$ be a pre-$\Delta$-monoid. For any null-sequence $\alpha_n,\beta_n\in \Omega(M,e)$, the following elements of $\pi_1(M,e)$ are all equal each other:
\begin{enumerate}
\item $\left[\prod_{n=1}^{\infty}\alpha_n\cdot\beta_n\right]$
\item $\left[\prod_{n=1}^{\infty}\alpha_n\ast\beta_n\right]$
\item $\left[\prod_{n=1}^{\infty}\alpha_n\right]\left[\prod_{n=1}^{\infty}\beta_n\right]$
\item $\left[\prod_{\tau}\alpha_n\cdot\beta_n\right]$
\item $\left[\prod_{\tau}\alpha_n\ast\beta_n\right]$
\item $\left[\prod_{\tau}\alpha_n\right]\left[\prod_{\tau}\beta_n\right]$
\end{enumerate}
\end{corollary}

\begin{proof}
The equality of (1),(2), and (3) was verified in Section \ref{subsectionelementarymonoids}. Applying Corollary \ref{basiccorollary} to the null-sequences $\alpha_n$, $\beta_n$, $\alpha_n\cdot\beta_n$, and $\alpha_n\ast\beta_n$ respectively gives homotopies $\prod_{\tau}\alpha_n\simeq \prod_{n=1}^{\infty}\alpha_n$, $\prod_{\tau}\beta_n\simeq \prod_{n=1}^{\infty}\beta_n$, $\prod_{\tau}\alpha_n\cdot\beta_n\simeq \prod_{n=1}^{\infty}\alpha_n\cdot\beta_n$, and $\prod_{\tau}\alpha_n\ast\beta_n\simeq \prod_{n=1}^{\infty}\alpha_n\ast\beta_n$. Combining these with the homotopies for (1)-(3) verifies the remaining equalities.
\end{proof}

\begin{example}
For $m\geq 2$, the space $\Omega^{m-1}(X,x)=(X,x)^{(S^{m-1},e_0)}$ of based maps on the $(m-1)$-sphere with basepoint $c_x:S^m\to X$ is based homotopy equivalent to the Moore loop space $\Omega^{\ast}(\Omega^{m-2}(X,x),c_x)$. Since being transfinitely $\pi_1$-commutative at the basepoint is invariant under basepoint-preserving homotopy equivalence (recall Lemma \ref{homequivlemma}), $\Omega^{m-1}(X,x)$ is transfinitely $\pi_1$-commutative at the constant map $c_x:S^m\to X$. We note that applying this infinitary commutativity in each variable of a map $(I^m,\partial I^m)\to (X,x)$ allows one to shuffle certain infinite arrangments of $m$-cubes similar that used in \cite{EK00higher}.
\end{example}

In a pre-$\Delta$-monoid $M$, transfinite $\pi_1$-commutativity is only guaranteed at the identity. We show how this result is strengthened when $M$ is a group.

\begin{theorem}\label{grouptheorem}
A pre-$\Delta$-group is transfinitely $\pi_1$-commutative at all of its points.
\end{theorem}

\begin{proof}
Let $G$ be a pre-$\Delta$-group. For given $g\in G$, we verify the factorization property given in Theorem \ref{transfinitecharacteriationtheorem}. Let $f:(\bbh,b_0)\to (G,g)$ be a map and pick any path $\alpha:(I,0,1)\to (G,e,g)$. Consider the map $g^{-1}\ast f:(\bbh,b_0)\to (G,e)$, which is continuous since $G$ is a pre-$\Delta$-group. The map $H:\bbh\times I\to G$, $k(x,t)=\alpha(t)\ast g^{-1}\ast f(x)$ is also continuous and defines a free homotopy from $H(x,0)=g^{-1}\ast f(x)$ to $H(x,1)=f(x)$ along the path $H(b_0,t)=\alpha(t)$. Thus if $\varphi_{\alpha}:\pi_1(G,g)\to \pi_1(G,e)$ is the basepoint-change isomorphism $\varphi_{\alpha}([\beta])=[\alpha\cdot\beta\cdot\alpha^{-}]$, we have $(g^{-1}\ast f)_{\#}=\varphi_{\alpha}\circ f_{\#}$. Since $g^{-1}\ast f:(\bbh,b_0)\to (G,e)$ is based at the identity and $G$ is a pre-$\Delta$-monoid, Corollary \ref{factorizationlemma} gives the existence of a unique homomorphism $F_{\alpha}:\pi_1(\bbt,x_0)\to (G,e)$ such that $F_{\alpha}\circ\eta_{\#}=(g^{-1}\ast f)_{\#}=\varphi_{\alpha}\circ f_{\#}$. Now $\lambda_{\alpha}=\varphi_{\alpha}^{-1}\circ F_{\alpha}$ is a homomorphism satisfying $\lambda_{\alpha}\circ\eta_{\#}= f_{\#}$. By Theorem \ref{transfinitecharacteriationtheorem}, we conclude that $G$ is transfinitely $\pi_1$-commutative at $g$.
\end{proof}

\subsection{Slide homotopies}

In this subsection, we observe another interesting phenomenon that occurs in monoids, namely, that we may begin with a transfinite product $\prod_{\tau}\alpha_n\in \Omega(M,e)$ and ``slide" individual factors $\alpha_n$ along a given path $\beta:(I,0)\to (M,e)$ so that the result remains homotopic to $\left(\prod_{\tau}\alpha_n\right)\cdot\beta$. Quite remarkably, the resulting homotopy class is independent of the position along $\beta$ to which we have chosen to slide each $\alpha_n$.

\begin{lemma}[Elementary Slide]\label{slidelemma1}
For any loop $\alpha\in \Omega(M,a)$ and path $\beta:\ui\to M$ from $e$ to $b$, we have
\begin{enumerate}
\item $\alpha\simeq(\beta\ast a)\cdot (b\ast\alpha)\cdot(\beta\ast a)^{-}$ by a homotopy in $Im(\beta)\ast Im(\alpha)$,
\item $\alpha\simeq(a\ast\beta)\cdot (\alpha\ast b)\cdot(a\ast\beta)^{-}$ by a homotopy in $Im(\alpha)\ast Im(\beta)$.
\end{enumerate}
Moreover, if $a=e$, then $[b\ast\alpha]=[\alpha\ast b]$ in $\pi_1(M,b)$.
\end{lemma}

\begin{proof}
Viewing $\alpha$ as a map $(S^1,x_0)\to (M,a)$ and consider the map $H: S^1\times I\to M$ defined by $H(s,t)=\beta(t)\ast\alpha(s)$. Note that $H(s,0)=\alpha(s)$ and $H(s,1)=b\ast \alpha$. Define $\iota:\ui\to S^1\times \ui$ by $\iota(t)=(x_0,t)$ and note $H\circ \iota(t)=H(x_0,t)=\beta(t)\ast a$. Composing $H$ with the homotopy $(s,0)\simeq\iota \cdot (s,1)\cdot \iota^{-}$ in $S^1\times I$, gives the desired homotopy $\alpha\simeq(\beta\ast a)\cdot (b\ast\alpha)\cdot(\beta\ast a)^{-}$, which has image in $Im(\beta)\ast Im(\alpha)$. Applying the same argument with the map $H': S^1\times I\to M$ defined by $H'(s,t)=\alpha(s)\ast \beta(t)$ gives the homotopy in (2). Finally, when $a=e$, (1) and (2) give $b\ast \alpha\simeq \alpha\ast b$.
\end{proof}

\begin{proposition}\label{hhausprop}
If $M$ does not have well-defined infinite $\pi_1$-products at $e$, then $M$ does not have well-defined infinite $\pi_1$-products at any of its points.
\end{proposition}

\begin{proof}
Suppose $f_1,f_2:(\bbh,b_0)\to (M,e)$ are maps such that $f_1\circ\ell_n\simeq f_2\circ\ell_n$ for all $n\in\bbn$ and $f_1\circ\ell_{\infty}\nsimeq f_2\circ\ell_{\infty}$. Pick any point $b\in M$ and path $\beta:(\ui,0,1)\to (M,e,b)$. Given any loop $L\in \Omega(\bbh,b_0)$, Lemma \ref{slidelemma1} gives $(f_i\circ L)\simeq \beta\cdot (b\ast (f_i\circ L))\cdot \beta^{-}$. Therefore, $(b\ast (f_1\circ L))\simeq (b\ast (f_2\circ L))$ if and only if $f_1\circ L\simeq f_2\circ L$. In particular, $f_1\circ\ell_n\simeq f_2\circ\ell_n$ $\Rightarrow$ $(b\ast (f_1\circ \ell_n))\simeq (b\ast (f_2\circ \ell_n))$ and $f_1\circ\ell_{\infty}\nsimeq f_2\circ\ell_{\infty}$ $\Rightarrow$ $(b\ast (f_1\circ \ell_{\infty}))\nsimeq (b\ast (f_2\circ \ell_{\infty}))$. It follows that $\{b\ast (f_1\circ \ell_n)\}$ and $\{b\ast (f_1\circ \ell_n)\}$ are null-sequences of term-wise homotopic loops in $\Omega(M,b)$, for which the corresponding infinite concatenations are not homotopic. Thus $M$ is does not have well-define infinite $\pi_1$-products at $b$.
\end{proof}

We exhibit an example of the phenomenon described in Proposition \ref{hhausprop} in Example \ref{exampleha}.

Lemma \ref{slidelemma1} describes a sliding of $\alpha$ along a translation of $\beta$, which becomes a true sliding along $\beta$ when $a=e$. In this case, $\alpha\in \Omega(M,e)$ is the homotopy path-conjugate (by $\beta$) of the loops $b\ast\alpha,\alpha\ast b$. We show that this kind of sliding can occur in a transfinite fashion for a sequence of loops based at $e$, where loops are dropped off wherever one wishes along $\beta$.

Let $I_n\in\mci(\mcc)$ denote the unique component of $I\backslash\mcc$ such that $\ell_{\tau}|_{\ov{I}_n}\equiv \ell_n$. Let $\Gamma:I\to I$ denote the ternary Cantor map, which identifies the closure of each set $I_n\in \mci(\mcc)$ to a point and is bijective elsewhere. Take $k_n:\ov{I_n}\to\ui$, $n\in\bbn$ to be the usual increasing linear homeomorphism.

\begin{definition}
Let $\alpha_n\in \Omega(M,e)$ be a null-sequence and $\beta:(I,0,1)\to (M,e,b)$ be a path. The \textit{transfinite slide of $\{\alpha_n\}$ along $\beta$} is a the path $\mathcal{S}_{\beta}^{\alpha_n}:(I,0,1)\to (M,e,b)$ defined as follows: 
\[\mathcal{S}_{\beta}^{\alpha_n}(t)=\begin{cases}
\beta(\Gamma(t)), & \text{ if }t\in\mcc\\
\beta(\Gamma(t))\ast\alpha_n(k_n(t)), & \text{ if }t\in I_n.
\end{cases}\]
\end{definition}

\begin{figure}[b]
\centering \includegraphics[height=1.5in]{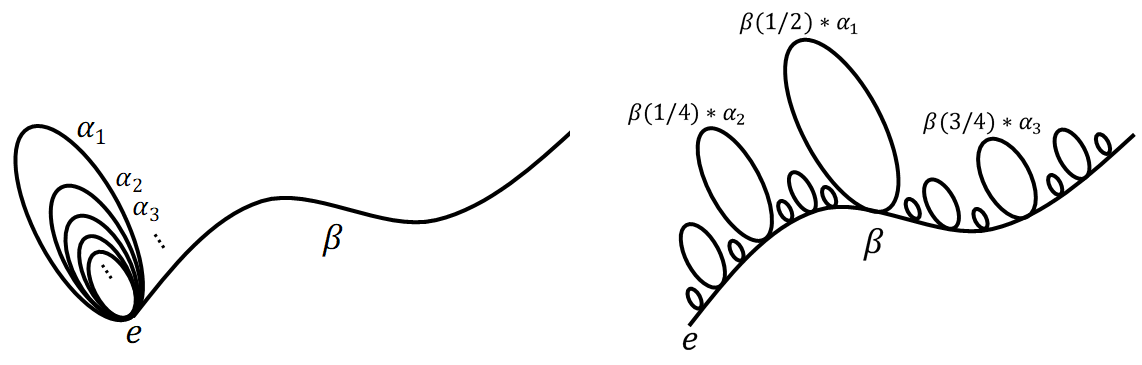}
\caption{\label{slidefig}The null-sequence $\{\alpha_n\}$ and path $\beta$ (left) and the transfinite slide $\mathcal{S}_{\beta}^{\alpha_n}$ (right).}
\end{figure}

\begin{lemma}\label{slidelemmatech}
For any null sequence $\alpha_n\in\Omega(M,e)$ and path $\beta:(I,0,1)\to (M,e,b)$, we have $\mathcal{S}_{\beta}^{\alpha_n}\simeq\left(\prod_{\tau}\alpha_n\right)\cdot\beta$.
\end{lemma}

\begin{proof}
Let $f:(\bbh,b_0)\to (M,e)$ be the map such that $f\circ \ell_n=\alpha_n$ and define $H:\bbh\times I\to M$ by $H(x,s)=\beta(s)\ast f(x)$. Notice that $H(b_0,s)=\beta(s)$ and $H(x,0)=f(x)$. Let $\iota:\ui\to \bbh\times I$ the path $\iota(s)=(b_0,s)$ and for each $s\in \ui$, define $\ell_{n}^{s}:I\to C_{n}\times \{s\}$ by $\ell_{n}^{s}(t)=(\ell_{n}(t),s)$. Then we have $H\circ\left(\prod_{\tau}\ell_{n}^{0}\right)\cdot\iota=H\circ\left(\prod_{\tau}\ell_{n}^{0}\right)\cdot H\circ\iota=\left(\prod_{\tau}\alpha_n\right)\cdot\beta$.

Next, define $\gamma:(I,0,1)\to (\bbh\times I,(b_0,0),(b_0,1))$ by \[\gamma(t)=\begin{cases}
\iota(\Gamma(t)), & \text{ if }t\in\mcc\\
\ell_{n}^{\Gamma(t)}(k_n(t)), & \text{ if }t\in I_n
\end{cases}\]and notice that $H\circ \gamma=\mathcal{S}_{\beta}^{\alpha_n}$. It suffices to show that $\left(\prod_{\tau}\ell_{n}^{0}\right)\cdot\iota$ and $\gamma$ are path-homotopic in $\bbh\times I$. Consider the deformation retraction $R:(\bbh\times I)\times I\to \bbh\times I$, $R(x,s,t)=(x,st)$ onto $\bbh\times \{0\}$. Then we have $\gamma(t)=R(\gamma(t),1)\simeq R(\gamma(t),0)\cdot\iota=\left(\prod_{\tau}\ell_{n}^{0}\right)\cdot\iota$, proving the claim.
\end{proof}

\begin{theorem}[Transfinite Slide]\label{slidetheorem}
For any null sequence $\alpha_n\in\Omega(M,e)$, homotopic paths $\beta,\beta ':(I,0,1)\to (M,e,b)$, and bijection $\phi:\bbn\to\bbn$, we have $\mathcal{S}_{\beta}^{\alpha_n}\simeq\mathcal{S}_{\beta '}^{\alpha_{\phi(n)}}$.
\end{theorem}

\begin{proof}
By Lemma \ref{slidelemmatech}, we have $\mathcal{S}_{\beta}^{\alpha_n}\simeq\left(\prod_{\tau}\alpha_n\right)\cdot\beta$ and $\mathcal{S}_{\beta '}^{\alpha_{\phi(n)}}\simeq\left(\prod_{\tau}\alpha_{\phi(n)}\right)\cdot\beta '$. We have $\prod_{\tau}\alpha_n\simeq \prod_{\tau}\alpha_{\phi(n)}$ by transfinite $\pi_1$-commutativity and since we have assumed $\beta\simeq\beta '$, the result follows.
\end{proof}
\begin{remark}
In the above results, we have chosen to slide each $\alpha_n=f\circ\ell_n$ to a unique position along $\beta$. Using transfinite commutativity, one can easily modify this so that any subset of the factors $\{\alpha_n\mid n\in\bbn\}$ slides to the same point along $\beta$. In particular, given any infinite partition $\mathcal{P}=\{P_1,P_2,P_3,\dots\}$ of $\bbn$ and ordering on each $P_m$, define $g:\bbh\to \bbh$ so that $g\circ\ell_m$ is the (finite or infinite) concatenation $\prod_{n\in P_m}\ell_n$. Applying the above results to the sequence $\{f\circ g\circ \ell_m\}$ and path $\beta$ gives the analogues where the loops $\{\alpha_n\mid n\in P_m\}$ slide to the same point on $\beta$ and appear as consecutive subloops of $\mathcal{S}_{\beta}^{f\circ g\circ \ell_m}$. Additionally, by replacing the Cantor map $\Gamma$ with a carefully chosen, non-decreasing, continuous, surjection $\ui\to\ui$, one may slide the loops $\alpha_n$ to any given countable set of points along the parameterization of $\beta$.
\end{remark}

\section{The James Reduced Product}\label{sectionjrproduct}

Let $X$ be a path-connected space with fixed basepoint $e\in X$. The \textit{James Reduced Product} on the based space $(X,e)$ is the set $\jx=\coprod_{n\in\bbn}X^n/\mathord{\sim}$ with the relation generated by $(x_1,\dots ,x_{j-1},e,x_{j+1},\dots,x_n)\sim (x_1,\dots ,x_{j-1},x_{j+1},\dots,x_n)$. We write elements of $J(X,e)$ as finite words $w=x_1x_2\cdots x_n$ in the alphabet $X\backslash\{e\}$ where $|w|=n$ denotes the length of the word. The basepoint represents the empty word $e$ which has length $|e|=0$. Under the operation of word concatenation, $J(X)$ is the free monoid on the based set $(X,e)$. Let $q:\coprod_{n\in\bbn}X^n\to \jx$ be the natural surjection, $J_n(X)=q(X^n)$ be the set of words with length $\leq n$, and $q_n:X^n\to \jnx$ be the restriction of $q$. We give $\jx$ the topology that is typically used in algebraic topology \cite[Chapter VII.2]{WhiteheadEOH}: $X^n$ has the product topology, $\jnx$ has the quotient topology with respect to $q_n$, and $J(X)$ has the weak topology with respect to the subspaces $\{\jnx\mid n\in\bbn\}$ so that $C\subseteq \jx$ is closed (resp. open) if and only if $C\cap \jnx$ is closed (resp. open) for all $n\in\bbn$. In other words, $\jx=\varinjlim_{n}\jnx$ is the inductive limit of the subspaces $\jnx$. We let $\sigma:X\to \jx$ denote the continuous injection of generators and note that every based map $f:X\to M$ to a topological monoid $M$ extends uniquely to a continuous homomorphism $\wt{f}:\jx\to M$ such that $\wt{f}\circ\sigma =f$.

\begin{remark}
In our analysis of $J(X)$, it is helpful to think of $J(X)$ in terms of the following inductive construction: begin with $X=J_1(X)$ and glue $X^2$ to $X$ by considering $X\times\{e\}$ and $\{e\}\times X$ as two sides of a square representing $X^2$ and identifying these two faces homeomorphically with $X$ to obtain $J_2(X)$. Inductively, visualize $X^n$ as an $n$-cube $[0,1]^n$ where the $n$ faces on the coordinate axes correspond to the subspaces $\{e\}\times X^{n-1},X\times \{e\}\times X^{n-2},\dots, X^{n-1}\times \{e\}$, which are homeomorphic to $X^{n-1}$. Now $J_n(X)$ is obtained by attaching $X^n$ to $J_{n-2}(X)$ by gluing each of these faces to $J_{n-1}(X)$ according the identifications given by $q_{n-1}$.
\end{remark}

\begin{remark}\label{fretopmonoidremark}
Although $\jx$ has the universal property of the free topological monoid $TM(X)$ on a based space $(X,e)$, the word-concatenation operation $\jx\times\jx\to \jx$ can fail to be continuous. For example, the proof of \cite[Theorem 4.1]{FOT} may be used to show that $J(\bbq)\times J(\bbq)\to J(\bbq)$ is not continuous. For a path-connected example, one may use the cone $C\bbq=\frac{\bbq\times I}{\bbq\times\{1\}}$ with the image of $(0,0)$ as the basepoint $e$. In contrast, the free topological monoid $TM(X)$ has the finest topology such that $TM(X)\times TM(X)\to TM(X)$ is continuous and $\sigma:X\to TM(X)$ is continuous and universal in the same sense as it is for $\jx$. Thus we have $\jx=TM(X)$ if and only if $\jx\times\jx\to \jx$ is continuous. We verify this equality in the important case of $k_{\omega}$-spaces in Lemma \ref{komegalemma}.
\end{remark}

\subsection{The topology of $\jx$}

Since we consider James reduced products in significant generality, we take care with the topological properties of $J(X)$. Every result here either makes up the argument for Theorem \ref{surjectivitytheorem} or is needed for a computation in Section \ref{computationsection}.

\begin{proposition}\label{quotientmapprop1}
Consider $J(X)$ for any space $X$ and $e\in X$. Then
\begin{enumerate}
\item the canonical map $q:\coprod_{n\in\bbn}X^n\to J(X)$ is a quotient map,
\item if $\{e\}$ is closed in $X$, then $\jnx$ is closed in $\jx$ for all $n\in\bbn$,
\item if $X$ is $T_1$, then $J_n(X)$, $n\in\bbn$ and $J(X)$ are $T_1$.
\end{enumerate}
\end{proposition}

\begin{proof}
(1) Let $C\subseteq J(X)$ such that $q^{-1}(C)$ is closed. Then $q^{-1}(C)\cap X^n=q_{n}^{-1}(C\cap J_n(X))$ is closed in $X^n$ for all $n\in\bbn$. Since $q_n$ is quotient, $C\cap J_n(X)$ is closed in $J_n(X)$ for all $n\in\bbn$ and thus $C$ is closed in $J(X)$.

(2) By (1), it suffices to show that $\coprod_{i\in\bbn}X^i\setminus q^{-1}(\jnx)$ is open in $\coprod_{i\in\bbn}X^i$. Let $z\in \coprod_{i\in\bbn}X^i\setminus q^{-1}(J_n(X))$. Then $z=(z_1,z_2,\dots ,z_{n+k})$ where $k\geq 1$ and $z_i\neq e$ for at least $(n+1)$-many $i$. Set $U_i=X\setminus e$ if $z_i\neq e$ and $U_i=X$ if $z_i=e$. Then $U=\prod_{i=1}^{n+k}U_i$ is an open neighborhood of $z$ and $U\subseteq \coprod_{i\in\bbn}X^i\setminus q^{-1}(J_n(X))$ since for all $u\in U$, we have $|q(u)|>n$.

(3) if $X$ is $T_1$, then so is $X^n$. Each fiber $q_{n}^{-1}(w)$, $w\in \jnx$ is finite and therefore closed in $X^n$. Thus $\{w\}$ is closed in $J_n(X)$ for all $n\in\bbn$. Since $\jx$ is the inductive limit of $T_1$ spaces, $\jx $ is $T_1$.
\end{proof}

\begin{proposition}\label{sequentiallycompactprop}
If $X$ is $T_1$, then every sequentially compact subspace of $J(X)$ lies in $J_n(X)$ for some $n\in\bbn$. In particular, if $f:Y\to \jx$ is a map from a sequentially compact space $Y$, then $Im(f)\subseteq \jnx$ for some $n\in\bbn$.
\end{proposition}

\begin{proof}
Let $Y$ be a sequentially compact subspace of $J(X)$ and suppose, to the contrary, that for all $n\in\bbn$, there exists $y_n\in J_n(X)\cap (Y\backslash J_{n-1}(X))$. Find a convergence subsequence $\{y_{n_k}\}\to y$ in $Y$ and $m\in\bbn$ such that $y\in J_{m}(X)$. Let $C=\{y_{n_k}\mid n_k> m\}$ and notice that $y\notin C$. Since $C\cap J_{n}(X)$ is finite in the $T_1$ space $J_n(X)$ (recall (3) of Proposition \ref{quotientmapprop1}), $C\cap J_{n}(X)$ is closed in $\jnx$ for all $n\in\bbn$. Thus $C$ is closed is $J(X)$; a contradiction.
\end{proof}

To proceed further, we must describe basic open neighborhoods of a point $w\in\jnx$ supposing that $\{e\}$ is closed in $X$. Notice that $\jnx\backslash J_{n-1}(X)\cong (X\backslash e)^n$ where $(X\backslash e)^n$ is open in $X^n$. Hence, if $w=x_1x_2\cdots x_n$ has length $n$, then neighborhoods of the form $q_n(\prod_{i=1}^{n}U_i)$ where $U_i$ is an open neighborhood of $x_i$ in $X\backslash \{e\}$ form a neighborhood base at $w$. Fix $w\in \jnx$ where $|w|=m<n$. The fiber $q_{n}^{-1}(w)$ consists of all $n$-tuples $v=(x_1,x_2,\dots ,x_n)$ such that there exists an $m$-element subset $F_v=\{i_1,i_2,i_3,\dots ,i_m\}\subseteq \{1,2,3,\dots ,n\}$ such that $q_n(v)=w=x_{i_1}x_{i_2}x_{i_3}\cdots x_{i_m}$ and $x_j=e$ if $j\notin F_v$. Let $U_1,U_2,U_3,\dots, U_m,V$ be open neighborhoods in $X$ such that $x_{i_j}\in U_j$ and $e\in V$. Since $\{e\}$ is closed in $X$, we may choose these neighborhoods so that $e\notin U_j$ for all $j\in\{1,2,\dots,m\}$. For each $v\in q_{n}^{-1}(w)$, consider the open neighborhood $N_v=\prod_{k=1}^{n}A_{k}^{v}$ of $v$ in $X^n$ where $A_{k}^{v}=V$ if $k\notin F_v$ and $A_{k}^{v}=U_j$ if $k=i_j$. Therefore, \[N:=N(U_1,U_2,U_3,\dots,U_m,V)=\bigcup \{N_v\mid q(v)=w\}\] is an open set containing $q_{n}^{-1}(w)$ such that if $y\in N$, then $|q_n(y)|\geq m$. 

\begin{proposition}
For $|w|=m<n$, the set $N=\bigcup \{N_v\mid q(v)=w\}$ defined above is saturated with respect to $q_n$.
\end{proposition}

\begin{proof}
Suppose $y=(y_1,y_2,y_3,\dots, y_n)\in X^n$, $a=(a_1,a_2,a_3,\dots,a_n)\in N_v$, and $q_n(y)=y_{l_1}y_{l_2}y_{l_3}\cdots y_{l_{m'}}=q_n(a)$ for $\{l_1,l_2,l_3,\dots, l_{m'}\}\subseteq \{1,2,3,\dots ,n\}$. Recall that for each $j\in \{1,2,3,\dots ,m\}$, we have $e\neq a_{i_j}\in U_j$ and $a_k\in V$ if $k\neq i_j$ for any $j$. Thus the letters $a_{i_1},a_{i_2},a_{i_3},\dots a_{i_m}$ appear (not necessarily in a unique position) in order within $y_{l_1}y_{l_2}y_{l_3}\cdots y_{l_{m'}}$. Let $F'\subseteq \{l_1,l_2,l_3,\dots, l_{m'}\}$ be an $m$-element set corresponding to such a sequence. Observe that if $k=l_r\notin F'$, then $y_{l_r}=a_k\in V$ and if $k\neq l_r$ for any $r\in\{1,2,\dots,m'\}$, then $y_k=e$. Finally, let $v'\in q_{n}^{-1}(w)$ be the $n$-tuple having the elements $x_{i_1},x_{i_2},\dots, x_{i_m}$ appear in order in the coordinates with indices $F'$ and $e$ elsewhere. Now it is clear that $y\in N_{v'}$, completing the proof.
\end{proof}

The continuity of $q_n$ makes it clear that neighborhoods $q_n(N)$ constructed above (i.e. determined by the open sets $U_1,U_2,U_3,\dots, U_m,V$ in $X$) form a neighborhood base at a word $w\in J_n(X)$ of length $|w|<n$. We call an open set of the form $q_n(N)$ a \textbf{standard neighborhood of }$w$ in $\jnx$. Simply put, $q_n(N)$ consists of all words $y_1y_2\dots y_r$ such that $y_1,y_2,\dots ,y_r$ has a subsequence $y_{i_1},y_{i_2},\dots,y_{i_m}$ with $y_{i_j}\in U_j$ and such that $y_i\in V$ if $i\notin \{i_1,i_2,\dots i_m\}$.

\begin{lemma}\label{agreeingtopologieslemma}
If $\{e\}$ is closed in $X$, then the quotient topology on $\jnx$ agrees with the subspace topology inherited from $\jx$. In particular, the inclusion of generators $\sigma:X\to \jx$ is an embedding.
\end{lemma}

\begin{proof}
Since $q_n:X^n\to \jx$ is continuous, the quotient topology on $\jnx$ is finer than the subspace topology. First, we show the two topologies agree at points in $\jnx\backslash J_{n-1}(X)$. If $w=x_1x_2\cdots x_n$ has length $n$, then a basic neighborhood of $w$ in the quotient topology on $\jnx$ is of the form $\mathscr{U}=q_{n}(\prod_{j=1}^{n}U_j)$ where $U_i$ is an open neighborhood of $x_i$ in $X\backslash\{e\}$. Let $\mathscr{V}$ be the subset of $\jx$ consisting of all words $y_1y_2\cdots y_k$ such that $y_1,y_2,\dots, y_k$ has a subsequence $y_{i_1},y_{i_2},\dots ,y_{i_n}$ with $y_{i_j}\in U_j$ for all $j$. Now, $q^{-1}(\mathscr{V})$ is the union of the open sets of the form $\prod_{i=1}^{k}B_k$ where $U_1,U_2,\dots ,U_n$ appears as a subsequence of $B_1,B_2,\dots ,B_k$ and such that $B_i=X$ if $B_i$ is not a member of this subsequence. Since $\mathscr{V}$ is open in $\jx$ and $\mathscr{V}\cap \jnx=\mathscr{U}$, it follows that $\mathscr{U}$ is in the subspace topology on $\jnx$.

If $|w|=m<n$, consider an open set $N=N(U_1,U_2,\dots,U_m,V)$ in $X^n$ as constructed above so that $\mathscr{W}_n=q_n(N)$ is a standard neighborhood of $w$ in $\jnx$. For all $k>n$, we may use the same open sets $U_1,U_2,\dots,U_m,V$ in $X$ to form the corresponding open set $N_k$ in $X^k$ so that $\mathscr{W}_k=q_k(N_k)$ is a standard neighborhood of $w$ in $J_k(X)$. Note that $\mathscr{W}_n\subseteq \mathscr{W}_{n+1}\subseteq \mathscr{W}_{n+2}\subseteq \cdots$ and $\mathscr{W}_{k+1}\cap J_k(X)=\mathscr{W}_k$ for all $k\geq n$. Therefore, if $\mathscr{W}=\bigcup_{k\geq n}\mathscr{W}_k$, then $\mathscr{W}\cap J_k(X)=\mathscr{W}_k$ is open in $J_k(X)$ for all $k\geq n$. Since the inclusion functions $J_k(X)\to J_n(X)$, $1\leq k<n$ are continuous with respect to the quotient topology on both domain and codomain, it follows that $\mathscr{W}\cap J_{k}(X)=\mathscr{W}_n\cap J_k(X)$ is open in $J_k(X)$ for $1\leq k<n$. Thus $\mathscr{W}$ is open in $\jx$ and $\mathscr{W}\cap J_n(X)=\mathscr{W}_n$ is in the subspace topology on $\jnx$.
\end{proof}

\begin{remark}
Assuming that $X$ is Hausdorff and $|w|=m<n$ as above, we may choose the open sets $U_1,U_2,\dots,U_m,V$ defining a standard neighborhood $q_n(N)$ to be pairwise disjoint or equal. With this choice, if $v,v'\in q_{n}^{-1}(w)$ and $v\neq v'$, then $N_v\cap N_{v'}=\emptyset$. Hence, $N=N(U_1,U_2,\dots,U_m,V)$ is the disjoint union of $|q_{n}^{-1}(w)|$-many basic open sets $N_v$ in the product topology on $X^n$. This observation makes the proof of the following proposition straightforward.
\end{remark}

\begin{proposition}\label{hausdorffprop}
If $X$ is Hausdorff, then $\jnx$ is Hausdorff for each $n\in\bbn$.
\end{proposition}

The quotient maps $q_n$ will only be open in degenerate cases. Hence, to prove that $\jx$ is a pre-$\Delta$-monoid, we employ the following notion due to E. Michael.

\begin{definition}\cite{Michaelbiq}
A continuous surjection $f:A\to B$ is a \textit{biquotient map} if whenever $b\in B$ and $\scru$ is an open cover of $f^{-1}(b)$ by open sets in $A$, then finitely many sets $f(U)$ with $U\in\scru$ cover some neighborhood of $b$.
\end{definition}
 
\begin{proposition}
If $X$ is Hausdorff, then $q_n:X^n\to \jnx$ is biquotient for all $n\in \bbn$.
\end{proposition}

\begin{proof}
Let $w\in \jnx$ and $\scru$ be an open cover of $q_{n}^{-1}(w)$ by open sets in $X^n$. If $w=x_1x_2\cdots x_n$ has length $n$, then $q_{n}^{-1}(w)=\{v\}\subseteq (X\backslash e)^n$. Find $U\in \scru$ containing $v$ and an open neighborhood $W$ of $v$ in $U\cap (X\backslash e)^n$. Then $q_n(W)$ is an open neighborhood of $w$ contained in $q_n(U)$. If $|w|<n$ and we write $w=x_{i_1}x_{i_2}\cdots x_{i_m}\in\jnx$ as described above, then for each $v\in q_{n}^{-1}(w)$, find $M_v\in\scru$ such that $v\in M_v$. Find open neighborhoods $U_1,U_2,\dots ,U_m,V$ and define $N=\bigcup\{N_v\mid v\in q_{n}^{-1}(w)\}$ so that $q_n(N)$ is a standard neighborhood of $w$. Since $q_{n}^{-1}(w)$ is finite, we may choose the sets $U_1,U_2,\dots ,U_m,V$ so that $N_v\subseteq M_v$ for all $v\in q_{n}^{-1}(w)$. Then $q_{n}(N)$ is covered by the finitely many sets $q_n(M_v)$, $v\in q_{n}^{-1}(w)$.
\end{proof}

Every biquotient map is quotient and products of biquotient maps are biquotient \cite[Theorem 1.2]{Michaelbiq}. Therefore, for all $n,m\in\bbn$, the map $q_{n}\times q_{m}:X^n\times X^m\to J_n(X)\times J_m(X)$, $n,m\in\bbn$ is quotient. Let $h:X^n\times X^m\to X^{n+m}$ be the canonical homomorphism and $\mu_{n,m}:J_n(X)\times J_m(X)\to J_{n+m}(X)$ be the word-concatenation function. Since $\mu_{n,m}\circ (q_n\times q_m)=q_{n+m}\circ h$ where $q_n\times q_m$ is quotient, the map $\mu_{n,m}$ is continuous by the universal property of quotient maps. We use this fact in the next lemma.

\begin{lemma}\label{jxisapredeltamonoidlemma}
If $X$ is Hausdorff, then $J(X)$ is a pre-$\Delta$-monoid.
\end{lemma}

\begin{proof}
Let $\alpha,\beta:\ui\to \jx$ be continuous paths and $\alpha\ast\beta:I\to J(X)$ denote the product path. By Proposition \ref{sequentiallycompactprop}, we have $Im(\alpha)\cup Im(\beta)\subseteq J_n(X)$ for some $n\in\bbn$. Since $\mu_{n,n}:J_n(X)\times J_n(X)\to J_{2n}(X)$ is continuous, $\alpha\ast\beta=\mu_{n,n}\circ (\alpha,\beta):I\to J_{2n}(X)\to J(X)$ is continuous. 
\end{proof}

Even though $\jx$ is not always a true topological monoid, Lemma \ref{jxisapredeltamonoidlemma} ensures that this issue does not affect our ability to apply the results in Section \ref{sectionmonoids}.

\begin{definition}
A \textit{$k_{\omega}$-decomposition} of a space $X$ is a sequence $A_1\subseteq A_2\subseteq A_3\subseteq \cdots\subseteq X$ of compact Hausdorff subspaces of $X$ such that $X=\bigcup_{m\in\bbn}A_m$ and such that $X$ has the weak topology with respect to the collection $\{A_m\}_{m\in\bbn}$. If a space $X$ admits a $k_{\omega}$-decomposition, we call $X$ a $k_{\omega}$\textit{-space}.
\end{definition}

We refer to \cite{FranklinThomaskomega} for the basic theory of $k_{\omega}$-spaces. All Hausdorff $k_{\omega}$-spaces are $T_4$ and the class of $k_{\omega}$-spaces is closed under finite products and countable disjoint unions. Hausdorff quotients of $k_{\omega}$-spaces are $k_{\omega}$-spaces and if $\{A_m\}$ is a $k_{\omega}$-decomposition of $X$, then $\{A_{m}^{n}\}$ is a $k_{\omega}$-decomposition of $X^n$. 

\begin{lemma}\label{komegalemma}
Suppose $\{A_m\}$ is a $k_{\omega}$-decomposition of $X$. Then 
\begin{enumerate}
\item $\jx=TM(X)$,
\item $\jx=\varinjlim_{m}J(A_m)$,
\item $\{J_m(A_m)\}$ is a $k_{\omega}$-decomposition of $J(X)$.
\end{enumerate}
\end{lemma}

\begin{proof}
If $\{A_m\}$ is a $k_{\omega}$-decomposition of $X$, then $X^n$ is a $k_{\omega}$-space. Since $\jnx$ is a Hausdorff (recall Proposition \ref{hausdorffprop}) quotient of $X^n$, $\jnx$ is a $k_{\omega}$-space. Moreover, since $\jx$ has the weak topology with respect to the sequence of closed $T_4$ spaces $\{\jnx\}$, $\jx$ is $T_4$. As a Hausdorff quotient of the $k_{\omega}$ space $\coprod_{n\in\bbn}X^n$, we conclude that $\jx$ is a $k_{\omega}$-space. 

Since finite products of quotient maps of $k_{\omega}$-spaces of quotient maps \cite{Michaelbiq}, $q\times q$ is quotient and it follows from the universal property of quotient maps that word-concatenation $J(X)\times J(X)\to J(X)$ is continuous. Hence, if $X$ is a $k_{\omega}$-space, then $J(X)$ is a topological monoid and is therefore the free topological monoid $TM(X)$ on $(X,e)$ (recall Remark \ref{fretopmonoidremark}). 

Since $\{A_{m}^{n}\}$ is a $k_{\omega}$-decomposition of $X^n$, the inclusions induce a quotient map $p_n:\coprod_{m\in\bbn}A_{m}^{n}\to X^n$ for each $n\in\bbn$. If $q_{A_m}:\coprod_{n\in\bbn}A_{m}^{n}\to J(A_m)$ is the canonical quotient map, then the top and vertical maps in the following commutative square are quotient. 
\[\xymatrix{\coprod_{(m,n)\in\bbn^2}A_{m}^{n} \ar[drr]^{p} \ar[d]_-{\coprod_{m}q_{A_m}} \ar[rr]^-{\coprod_{n}p_n} && \coprod_{n\in\bbn}X^n \ar[d]^-{q} \\
\coprod_{m\in\bbn}J(A_m) \ar[rr]_-{(J(i_m))} && \jx}\]
It follows that the bottom map is quotient. We conclude that $\jx$ has the weak topology with respect to $\{J(A_m)\}$, i.e. $\jx= \varinjlim_{m}J(A_m)$. Finally, let $p=q\circ \coprod_{n}p_n$ and notice that if $B_m=\bigcup\{A_{i}^{j}\mid \max\{i,j\}\leq m\}$, then $\{B_m\}$ is a $k_{\omega}$-decomposition for $\coprod_{(m,n)\in\bbn^2}A_{m}^{n}$. Since $p(B_m)=J_m(A_m)$, $\{J_m(A_m)\}$ is a $k_{\omega}$-decomposition for the quotient space $\jx$.
\end{proof}

We finish this subsection by showing that the based homotopy type the of James reduced product is an invariant of based homotopy type.

\begin{lemma}\label{homotopylemma}
If $H:X\times I\to Y$ is a basepoint-preserving homotopy, i.e. such that $H(e_1,t)=e_2$ for all $t\in I$, then so is the function $G:J(X,e_1)\times I\to J(Y,e_2)$ given by $G(x_1x_2\cdots x_n,t)=H(x_1,t)\ast H(x_2,t)\ast\,\cdots\, \ast H(x_n,t)$. 
\end{lemma}

\begin{proof}
The function $H':\coprod_{n\in\bbn}X^n\times I\to \coprod_{n\in\bbn} Y^n$ defined as $H'((x_1,x_2,\dots, x_n),t)=(H(x_1,t),H(x_2,t),\dots, H(x_n,t))$ on the $n$-th summand is continuous. Fix $(x_1,x_2,\dots, x_n)\in X^n$, $t\in I$, and set $A=\{i\mid H(x_i,t)=e_2\}$. Since $H(e_1,t)=e_2$, we have $G(q_X((x_1,x_2,\dots ,x_n)),t)=\prod_{i\notin A}H(x_i,t)=q_Y(H(x_1,t),H(x_2,t),\dots, H(x_n,t))$. Thus $q_Y\circ H'= G\circ (q_X\times id_I)$. By a theorem of Whitehead \cite[Lemma 4]{whiteheadquotient}, $q_X\times id_I$ is a quotient map. Therefore, $G$ is continuous by the universal property of quotient maps. Since $G(e_1,t)=e_2$ for all $t\in I$, $G$ is a basepoint-preserving homotopy.
\end{proof}

In particular, if two maps $f,g:(X,x)\to (Y,y)$ are homotopic rel. basepoint, then the induced homomorphisms $J(f),J(g):(J(X),x)\to (J(Y),y)$ are also homotopic rel. basepoint. The following lemma is an immediate consequence.

\begin{lemma}\label{homotopyequivalenlemma}
If $f:(X,e_1)\to(Y,e_2)$ is a based homotopy equivalence, then so is $J(f):J(X,e_1)\to J(Y,e_2)$.
\end{lemma}


\subsection{The surjectivity of the natural homomorphism $\pi_1(X,e)\to \pi_1(\jx,e)$}

The continuous injection $\sigma:X\to J(X)$ induces a natural homomorphism $\sigma_{\#}:\pi_1(X,e)\to \pi_1(\jx,e)$. This section is dedicated to proving the following theorem.

\begin{theorem}\label{surjectivitytheorem}
If $X$ is path connected and Hausdorff, then the inclusion $\sigma:X\to J(X)$ induces a surjection $\sigma_{\#}:\pi_1(X,e)\to\pi_1(J(X),e)$.
\end{theorem}

To prove Theorem \ref{surjectivitytheorem} for such a large class of spaces, we must carefully decompose loops in $\jx$. Recall that since $\jx$ is a pre-$\Delta$-monoid (Lemma \ref{jxisapredeltamonoidlemma}), all results in Section \ref{sectionmonoids} are applicable. For the remainder of this section, we assume $X$ is path connected and Hausdorff.

\begin{lemma}\label{pathextension}
Given $\alpha:\ui\to\jnx$, if there exists map $\beta:(0,1)\to X^n$ such that $q_n\circ\beta=\alpha|_{(0,1)}$, then there exists a unique path $\wt\alpha:\ui\to X^n$ such that $q_n\circ\wt\alpha=\alpha$.
\end{lemma}

\begin{proof}
We show that $\beta$ extends uniquely to $[0,1)$; the extension at $1$ is similar.

Let $w=\alpha(0)$. If $|w|=n$, recall that $J_{n-1}(X)$ is closed in $J_n(X)$, and find $s>0$ such that $\alpha([0,s))\subseteq J_n(X)\backslash J_{n-1}(X)$. If $h$ is the inverse of the homeomorphism $(q_n)|_{(X\backslash \{e\})^n}:(X\backslash \{e\})^n\to J_n(X)\backslash J_{n-1}(X)$, then we have a unique, continuous extension $\wt{\alpha}|_{[0,s)}=h\circ \alpha|_{[0,s)}$.

Suppose that $|w|<n$ and find a standard open neighborhood $U=q_n(N)$ of $w$ where $|w|=m$ and $N:=N(U_1,U_2,U_3,\dots ,U_m,V)$ is the neighborhood of $w$ constructed above. In particular, $q_{n}^{-1}(U)$ is the disjoint union of the open neighborhoods $N_v$ of $v\in q_{n}^{-1}(w)$. Since $\alpha$ is continuous at $0$, there exists $s>0$ such that $\alpha([0,s))\subseteq U$. Since $\beta((0,s))$ is connected and lies in $q_{n}^{-1}(U)$, it must lie in $N_v$ for some unique $v\in q_{n}^{-1}(w)$. Define $\wt{\alpha}(0)=v$. Since $N$ is the disjoint union of the sets $N_v$, it is clear that this is the only possible definition if we desire $\wt{\alpha}$ to be continuous at $0$. To verify the continuity of $\wt{\alpha}$ at $0$, we maintain the value of $v$ and $s$ but consider a new standard neighborhood $q_n(N')$ at $w$.

Let $W$ be an open neighborhood of $v=(x_1,x_2,x_3,\dots ,x_n)$ in $X^n$. We may find $U_1',U_2',U_3',\dots U_m',V'$ be open neighborhoods in $X$, as above, so that the neighborhood $N_{v}'=\prod_{k=1}^{n}B_{k}^{v}$ (where $B_{k}^{v}=V'$ if $k\notin F_v$ and $B_{k}^{v}=U_j'$ if $k=i_j$) of $v$ is contained in $W\cap N_v$. Consider the saturated open set $N':=N(U_1',U_2',U_3',\dots U_m',V')=\bigcup\{N_{y}'\mid y\in q_{n}^{-1}(w)\} $ in $X^n$ whose image $q_n(N')$ is a standard open neighborhood of $w$ in $\jnx$. The continuity of $\alpha$ ensures $\alpha([0,t))\subseteq q_n(N')$ for some $s>t>0$. Since $N_{v}'\subseteq N_v$ and $\beta(0,s)\subseteq N_v$, we must have $\beta((0,t))\subseteq N_{v}'$. In particular, $\wt{\alpha}([0,t))=\{w\}\cup\beta((0,t))\subseteq N_{v}'\subseteq W$.
\end{proof}

\begin{corollary}\label{liftlemma}
If $\alpha:\ui\to \jnx$ is a path such that $\alpha^{-1}(\jnx\backslash J_{n-1}(X))=(0,1)$, then there exists a unique path $\wt{\alpha}:\ui\to X^n$ such that $q_n\circ \wt{\alpha}=\alpha$.
\end{corollary}

\begin{proof}
Since $q_n$ restricts to a homeomorphism $(X\backslash e)^n\to \jnx\backslash J_{n-1}(X)$, it is clear that there is a unique map $\beta:(0,1)\to (X\backslash e)^n$ such that $q_n\circ \beta=\alpha|_{(0,1)}$. We then have that $\beta$ extends uniquely to a path $\wt{\alpha}$ by Lemma \ref{pathextension}.
\end{proof}

It is important to note that if $\alpha,\beta:\ui\to J_n(X)$ are composable paths, i.e. $\alpha(1)=\beta(0)=w$ for some $w$, and $\alpha^{-1}(\jnx\backslash J_{n-1}(X))=\beta^{-1}(\jnx\backslash J_{n-1}(X))=(0,1)$, then the unique lifts $\wt{\alpha},\wt{\beta}:\ui\to X^n$ satisfy $\{\wt{\alpha}(1),\wt{\beta}(0)\}\subseteq q_{n}^{-1}(w)$ but may not themselves be composable.

\begin{corollary}\label{decomposition}
For any path $\alpha:\ui\to \jnx$ with $\alpha^{-1}(\jnx\backslash J_{n-1}(X))=(0,1)$, there exists unique paths $\alpha^{1},\alpha^{2},\alpha^{3},\dots,\alpha^{n}:\ui\to X=J_1(X)$ such that $\alpha=\bigast_{i=1}^{n}\alpha^{i}$.
\end{corollary}

\begin{lemma}\label{imageconvergencelemma}
Suppose $n\geq 2$ and $\alpha_k:\ui\to \jnx$ is a (not necessarily composable) sequence of paths such that $\{\alpha_k\}$ is null at $w\in J_n(X)$ and that $\alpha_k=\bigast_{i=1}^{n}\alpha_{k}^{i}$ with $\alpha_{k}^{i}:\ui\to X$ for all $k\in\bbn$. Then for every neighborhood $U$ of $w$ in $\jnx$, the set $\bigast_{i=1}^{n}Im(\alpha_{k}^{i})$ lies in $U$ for all but finitely many $k\in\bbn$.
\end{lemma}

\begin{proof}
Let $q_n(N)$ be a standard neighborhood of $w$ in $\jnx$ such that $q_n(N)\subseteq U$. Recall that $N$ is the disjoint union of $N_v$, $v\in q_{n}^{-1}(w)$ where $N_v$ is a $n$-fold product of open sets from the list $U_1,U_2,U_3,\dots,U_m,V$. Let $\wt{\alpha}_{k}=(\alpha_{k}^{1},\alpha_{k}^{2},\alpha_{k}^{3},\dots, \alpha_{k}^{n}):\ui\to X^n$ be a lift of $\alpha_k$. Since $\alpha_k$ is null at $w$, there exists $K\in\bbn$ such that $Im(\alpha_k)\subseteq q_n(N)$ for all $k\geq K$. Thus $Im(\wt{\alpha}_{k})\subseteq N$ for all $k\geq K$. In particular, if $k\geq K$, the connectedness of $Im(\wt{\alpha}_k)$ implies that $Im(\wt{\alpha}_{k})\subseteq N_{v_k}$ for some $v_k\in q_{n}^{-1}(w)$. Therefore, since $N_{v_k}$ has the form of a product of $n$-open sets in $X$, we have $\prod_{i=1}^{n}Im(\alpha_{k}^{i})\subseteq N_{v_k}$. We conclude that $\bigast_{i=1}^{n}Im(\alpha_{k}^{i})=q_n(\prod_{i=1}^{n}Im(\alpha_{k}^{i}))\subseteq q_n(N_{v_k})\subseteq q_n(N)\subseteq U$ for $k\geq K$.
\end{proof}

\begin{lemma}\label{type1removal}
Let $n\geq 2$ and $\alpha:\ui\to\jnx$ be a path such that $\alpha^{-1}(\jnx\setminus J_{n-1}(X))=(0,1)$. Let $\alpha^i:\ui\to X=J_1(X)$ be the unique paths such that $\alpha=\bigast_{i=1}^n \alpha^i$. Suppose there exist $j,k\in\{1,2,\dots, n\}$ with $j\neq k$, $\alpha^j(0)=e$, and $\alpha^k(1)=e$. Then $\alpha$ is homotopic to a path $\alpha':\ui\to J_{n-1}(X)$ by a homotopy with image $\bigast_{i=1}^n Im(\alpha^i)$.
\end{lemma}

\begin{proof}
The existence of the $\alpha^i$ is guaranteed by Corollary \ref{decomposition}. Suppose $\alpha^j(0)=e$ and $\alpha^k(1)=e$ with $j\neq k$. Define $\phi:\{1,2,\dots,n\}\to\{1,2,\dots,n\}$ to be a bijection such that $\phi(k)=1$ and $\phi(j)=n$. As described in Corollary \ref{product-to-concatentation}, we have that $\bigast_{i=1}^n \alpha^i \simeq \prod_{i=1}^n (\bigast_{m=1}^n \beta_{i,m})$ for paths
\[\beta_{i,m}=\begin{cases}
\alpha^{m}(0), & \text{ if }\phi(m)\in \{i,i+1,\dots ,n\}\\
\alpha^m, & \text{ if }\phi(m)=i\\
\alpha^{m}(1), & \text{ if }\phi(m)\in \{1,2,\dots, i-1\}\\
\end{cases}\]
by a homotopy with image in $\bigast_{i=1}^n Im(\alpha^i)$. For convenience, write $\gamma_{i}=\bigast_{m=1}^{n}\beta_{i,m}$. By our choice of $\phi$, $\alpha^{k}$ is the $k$-th factor of $\gamma_1$ and $\alpha^{j}$ is the $j$-th factor of $\gamma_n$. Therefore $\alpha^{k}(1)=e$ is the $k$-th factor of $\gamma_2,\gamma_3, \dots ,\gamma_n$ and $\alpha^{j}(0)=e$ is the $j$-th factor of $\gamma_1,\gamma_2,\dots,\gamma_{n-1}$. Therefore, $\alpha'=\prod_{i=1}^n \gamma_i$ has image in $J_{n-1}(X)$.
\end{proof}

\begin{lemma}\label{alltype1removal}
For any $n\geq 2$ and $\alpha\in\Omega(\jnx,e)$, there exists $\beta\in\Omega(J_n(X),e)$ such that $\alpha\simeq\beta$ and where $\beta$ satisfies the following: if $(a,b)$ is a connected component of $\beta^{-1}(\jnx\setminus J_{n-1}(X))$ and $\beta|_{[a,b]}=\bigast_{i=1}^n \beta|_{[a,b]}^i$, then $\beta|_{[a,b]}^j(0)=e$ if and only if $\beta|_{[a,b]}^j(1)=e$.
\end{lemma}

\begin{proof}
Let $\{(a_j,b_j)\mid j\in J\}$ be the set of connected components $\alpha^{-1}(\jnx\setminus J_{n-1}(X))$ with the property that when we write $\alpha|_{[a_j,b_j]}=\bigast_{i=1}^n\alpha|_{[a_j,b_j]}^i$ (by Corollary \ref{decomposition}), there exist $k,m\in\bbn$ with $k\neq m$, $\alpha^k(0)=e$, and $\alpha^m(1)=e$. Set $U=\bigcup_{j\in J}(a_j,b_j)$. For $j\in J$, define $H_j:[a_j,b_j]\times I\to J_n(X)$ to be the homotopy obtained from Lemma \ref{type1removal} so that $H_j(s,0)=\alpha|_{[a_j,b_j]}(s)$, $H_j(s,1)$ has image in $J_{n-1}(X)$, and $Im(H_j)\subseteq \bigast_{i=1}^n Im(\alpha|_{[a_j,b_j]}^i)$. Now define $H:I^2\to \jnx$ by $$H(s,t)=
\begin{cases}
H_j(s,t), & \textit{if } s\in [a_j,b_j]\\
\alpha(s), & \textit{if } s\notin U.
\end{cases}$$
The continuity of $\alpha$ and each $H_j$ guarantees that $H$ is continuous everywhere except at points in $C\times I$, where $C$ is the set of limit points in $\partial U$. Thus to show $H$ is continuous, it suffices to show that $H$ is continuous at points in $C\times I$. Let $(w,t)\in C\times I$ be such a point. Let $\{(x_n,y_n)\}_{n\in\bbn}$ be a sequence in $I^2$ which converges to $(w,t)$. Let $M=\{n\in\bbn\mid x_n\in U\}$. If $M$ is finite then there exists $N\in\bbn$ such that $H(x_n,y_n)=\alpha(x_n)$ for all $n\geq N$, which implies $\{H(x_n,y_n)\}_{n\in\bbn}$ converges to $H(w,t)=\alpha(w)$. So suppose that $M$ is infinite. Let $V$ be an open neighborhood of $H(w,t)$. Let $U_k$ be the component of $x_k$ in $U$ for $k\in M$. By Lemma \ref{imageconvergencelemma}, there exists $K\in\bbn$ such that $H(U_k\times I)\subseteq \bigast_{i=1}^n Im(\alpha|_{U_k}^i) \subseteq V$ for all $k\geq K$. Thus $H(x_k,y_k)\in V$ for all $k\geq K$. Thus $\{H(x_k,y_k)\}_{k\in M}$ converges to $H(w,t)$. Also, $\{H(x_n,y_n)\}_{n\in\bbn\setminus M}$ is either finite or converges to $H(w,t)$. Hence $\{H(x_n,y_n)\}_{n\in\bbn}$ converges to $H(w,t)$. Therefore $H$ is continuous.

Let $\beta(s)=H(s,1)$ and observe that $\beta$ satisfies the conditions of the lemma.
\end{proof}

\begin{lemma}\label{type2removal}
Suppose $n\geq 2$ and $\alpha:\ui\to\jnx$ is a path such that $\alpha^{-1}(\jnx\setminus J_{n-2}(X))=(0,1)$ and where $\alpha$ satisfies the following: if $C_j$ is the closure of a connected component of $\alpha^{-1}(\jnx\setminus J_{n-1}(X))$ and $\alpha|_{C_j}=\bigast_{i=1}^n \alpha|_{C_j}^i$, then $\alpha|_{C_j}^m(0)=e$ if and only if $\alpha|_{C_j}^m(1)=e$. Then there exists path $\wt\alpha:\ui\to X^{2n-1}$ such that $q_{2n-1}\circ\wt\alpha=\alpha$. Moreover, $\wt\alpha=(\wt\alpha^1,\wt\alpha^2,\dots,\wt\alpha^{2n-1})$ for paths $\wt\alpha^i:\ui\to X$ and $\alpha=\bigast_{i=1}^{2n-1}\wt\alpha^i$ is homotopic to a path $\delta:\ui\to J_{n-1}(X)$ by a homotopy with image in $\bigast_{i=1}^{2n-1}Im(\wt\alpha^i)$.
\end{lemma}

\begin{proof}
Suppose $\alpha:\ui\to\jnx$ satisfies the conditions in the statement of the lemma. Let $\mca=\{A_j\mid j\in J\}$ be the set of connected components of $\alpha^{-1}(\jnx\setminus J_{n-1}(X))$ and let $\mcc=\{\ov{A_j}\mid j\in J\}$. By Lemma \ref{decomposition}, we may write $\alpha|_{C_j}=\bigast_{i=1}^n\alpha|_{C_j}^i$ for paths $\alpha|_{C_j}^i:C_j\to X$. Additionally, by hypothesis, for each $C_j\in\mcc$, there exists unique $m_j\in\{1,2,\dots, n\}$ such that $\alpha|_{C_j}^{m_j}(0)=\alpha|_{C_j}^{m_j}(1)=e$. Let $\mcd=\{D_k\mid k\in K\}$ be the set of connected components of $(0,1)\setminus \bigcup\mca$. Note that each $D_k\in\mcd$ is either a point, a closed interval, or possibly an interval of the form $(0,a]$ or $[b,1)$. Since $\alpha(\bigcup\mcd)\subseteq J_{n-1}(X)\backslash J_{n-2}(X)$, and $J_{n-1}(X)\backslash J_{n-2}(X)\cong (X\backslash\{e\})^{n-1}$, we may write $\alpha|_{D_k}=\bigast_{i=1}^{n-1}\alpha|_{D_k}^i$ for unique maps $\alpha|_{D_k}^i:D_k\to X\backslash \{e\}$. For $i\in \{1,2,\dots,n\}$, define function $\gamma^{2i-1}:(0,1)\to X$ by
$$\gamma^{2i-1}(s)=
\begin{cases}
\alpha|_{C_j}^i(s), & \textit{if } s\in C_j \textit{ and } m_j=i\\
e, & \textit{otherwise}.
\end{cases}$$
For $i\in\{1,2,\dots, n-1\}$, define function $\gamma^{2i}:(0,1)\to X$ by
$$\gamma^{2i}(s)=
\begin{cases}
\alpha|_{D_k}^i(s), & \textit{if } s\in D_k\\
\alpha|_{C_j}^{i}(s), & \textit{if } s\in C_j \textit{ and } m_j> i\\
\alpha|_{C_j}^{i+1}(s), & \textit{if } s\in C_j \textit{ and } m_j\leq i.
\end{cases}$$
We verify the continuity of both $\gamma^{2i-1}$ and $\gamma^{2i}$ in a similar fashion. In both cases it suffices to check continuity at points in $\bigcup \mcd$. So let $t\in\bigcup\mcd$ and consider $\alpha(t)=x_1x_2\cdots x_{n-1}$. Note that $\gamma^{2i}(t)=x_i$ and $\gamma^{2i-1}(t)=e$. Let $V$ be an open neighborhood of $e$ in $X$ and let $U_1,U_2,\dots,U_{n-1}$ be open neighborhoods in $X$ of of $x_1,x_2,\dots,x_{n-1}$ respectively such that $U_1,U_2,\dots,U_{n-1},V$ are pairwise disjoint or equal. Construct standard open neighborhood $q_n(N)$ of $\alpha(t)$ where $N=N(U_1,U_2,\dots,U_{n-1},V)=\bigcup\{N_v\mid v\in q_{n}^{-1}(\alpha(t))\}$. Since $\alpha$ is continuous, there exists an open neighborhood $W$ of $t$ in $(0,1)$ such that $\alpha(W)\subseteq q_n(N)$. 

To see that $\gamma^{2i-1}(W)\subseteq V$, we check that for all $W_j=W\cap C_j$ where $m_j=i$ we have $\gamma^{2i-1}(W_j)\subseteq V$. When restricting to such $W_j$, we can lift $\alpha|_{W_j}$ to unique map $(\alpha|_{W_j}^1,\dots,\alpha|_{W_j}^n):W_j\to X^n$. Since $e\notin \bigcup_{r=1}^{n-1}U_r$, we must have that $(\alpha|_{W_j}^1,\dots,\alpha|_{W_j}^n)(W_j)\subseteq N_v$ where $N_v=U_1\times U_2\times\dots\times U_{i-1}\times V\times U_i\times \dots \times U_{n-1}$. Thus $\alpha^i(W_j)\subseteq V$ which shows that $\gamma^{2i-1}(W_j)\subseteq V$. We conclude that $\gamma^{2i-1}(W)\subseteq V$.

For the continuity of $\gamma^{2i}$, we show that $\gamma^{2i}(W)\subseteq U_i$. We begin by again checking that for all $W_j=W\cap C_j$ we have that $\gamma^{2i}(W_j)\subseteq U_i$. For $\alpha|_{W_j}$, its unique lift $(\alpha|_{W_j}^1,\dots,\alpha|_{W_j}^n) :W_j\to X^n$ has image in $N_v=U_1\times U_2\times\dots\times U_{m_j-1}\times V\times U_{m_j}\times \dots \times U_{n-1}$ due to the fact that $e\notin \bigcup_{r=1}^{n-1}U_r$. Thus if $i<m_j$, we have that $\alpha|_{W_j}^i(W_j)=\gamma^{2i}(W_j)\subseteq U_i$. Similarly if $i\geq m_j$, we have that $\alpha|_{W_j}^{i+1}(W_j)=\gamma^{2i}(W_j)\subseteq U_i$. It remains to show that for all $W'_k=W\cap D_k$ we have that $\gamma^{2i}(W'_k)=\alpha|_{D_k}^{i}(W_{k}')\subseteq U_i$. We have $q_{n-1}\circ (\alpha|_{D_k}^{1},\dots ,\alpha|_{D_k}^{n-1})(W_{k}')=\alpha|_{D_k}(W_{k}')\subseteq q_{n}(N)$. Note that since $e\notin \bigcup_{r=1}^{n-1}U_r$, we have $q_{n-1}^{-1}( q_n(N))=U_1\times U_2\times \dots \times U_{n-1}$ in $X^{n-1}$. Thus $(\alpha|_{D_k}^{1},\dots ,\alpha|_{D_k}^{n-1})(W_{k}')\subseteq U_1\times U_2\times \dots \times U_{n-1}$, completing the proof of continuity.

Let $\gamma:(0,1)\to X^{2n-1}$ be the map defined as $\gamma=(\gamma^1,\gamma^2,\dots,\gamma^{2n-1})$. By construction, we have that $q_{2n-1}\circ\gamma=\alpha|_{(0,1)}$. Using Lemma \ref{pathextension}, extend $\gamma$ to a path $\wt\alpha=(\wt\alpha^1,\wt\alpha^2,\dots,\wt\alpha^{2n-1}):\ui\to X^{2n-1}$ so that $\alpha=q_{2n-1}\circ\wt\alpha=\bigast_{i=1}^{2n-1}\wt{\alpha}^{i}$. The definition of $\gamma^{2i-1}$ makes it clear that $\wt{\alpha}^{2i-1}:I\to X$ must be a loop based at $e$. Since $\alpha(0)\in J_{n-2}(X)$, there exists $j\in\{1,2,\dots ,n-1\}$ such that $\wt{\alpha}^{2j}(0)=e$. Define $\phi:\{1,2,\dots,2n-1\}\to\{1,2,\dots,2n-1\}$ to be a bijection such that $\phi(2j)=2n-1$. As described in Corollary \ref{product-to-concatentation}, we have that $\bigast_{i=1}^{2n-1} \wt\alpha^i \simeq \prod_{i=1}^{2n-1} (\bigast_{m=1}^{2n-1} \beta_{i,m})$ by a homotopy with image in $\bigast_{i=1}^{2n-1} Im(\wt\alpha^i)$. Set $\delta_i=\ast_{m=1}^{2n-1}\beta_{i,m}$. For $i\in\{1,2,\dots,2n\}$ we have that $\wt\alpha^{2j}(0)=e$ appears in the factorization of $\delta_i$. Additionally, for $(n-1)$-many values of $k\in\{1,2,\dots,n\}$, either $\wt\alpha^{2k-1}(0)=e$ or $\wt\alpha^{2k-1}(1)=e$ appear in the factorization of $\delta_i$. Finally, when $i=2n-1$, there are $n$-many occurrences of $\wt\alpha^{2k-1}(0)=e$ or $\wt\alpha^{2k-1}(1)=e$ in the factorization of $\delta_i$. Thus by our choice of $\phi$, for all $i\in\{1,2,\dots,2n-1\}$ there exist at least $n$-many values of $m$ such that $\beta_{i,m}=e$. We conclude that $\delta=\prod_{i=1}^{2n-1}\delta_i$, our desired path, has image in $J_{n-1}(X)$.
\end{proof}

\begin{lemma}\label{alltype2removal}
Let $n\geq 2$ and $\beta\in\Omega(J_n(X),e)$ such that $\beta$ satisfies the following: if $(a,b)$ is a connected component of $\beta^{-1}(\jnx\setminus J_{n-1}(X))$ and $\beta|_{[a,b]}=\bigast_{i=1}^n \beta|_{[a,b]}^i$, then $\beta|_{[a,b]}^j(0)=e$ if and only if $\beta|_{[a,b]}^j(1)=e$. Then there exists $\gamma\in\Omega(J_{n-1}(X),e)$ with $\beta\simeq\gamma$.
\end{lemma}

\begin{proof}
Suppose $\beta\in\Omega(\jnx,e)$ satisfies the conditions in the lemma. Let $U=\bigcup_{j\in J}(a_j,b_j)$ where $(a_j,b_j)$ is a connected component of $\beta^{-1}(\jnx\setminus J_{n-2}(X))$. For $j\in J$, define $H_j:[a_j,b_j]\times I\to J_n(X)$ to be the homotopy obtained from Lemma \ref{type2removal} so that $H_j(s,0)=\beta|_{[a_j,b_j]}(s)$ and where $H_j(s,1)$ has image in $J_{n-1}(X)$. Now define $H:I^2\to \jnx$ by $$H(s,t)=
\begin{cases}
H_j(s,t), & \textit{if } s\in [a_j,b_j]\\
\beta(s), & \textit{if } s\notin U.
\end{cases}$$

To show $H$ is continuous, it suffices to show that $H$ is continuous at points in $C\times I$, where $C$ is the set of limit points in $\partial U$. Let $(w,t)\in C\times I$ be such a point. Let $\{(x_n,y_n)\}_{n\in\bbn}$ be a sequence in $I^2$ which converges to $(w,t)$. Let $M=\{n\in\bbn\mid x_n\in U\}$. If $M$ is finite then there exists $N\in\bbn$ such that $H(x_n,y_n)=\beta(x_n)$ for all $n\geq N$, which implies $\{H(x_n,y_n)\}_{n\in\bbn}$ converges to $H(w,t)$. So suppose that $M$ is infinite. Let $V$ be an open neighborhood of $H(w,t)$. Let $U_k$ be the component of $x_k$ in $U$ for $k\in M$. As described in Lemma \ref{type2removal}, $\beta|_{U_k}=\bigast_{i=1}^{2n-1}\wt\beta|_{U_k}^i$. By Lemma \ref{imageconvergencelemma}, there exists $K\in\bbn$ such that $H(U_k\times I)\subseteq \bigast_{i=1}^{2n-1} Im(\wt\beta|_{U_k}^i) \subseteq V$ for all $k\geq K$. Thus $H(x_k,y_k)\in V$ for all $k\geq K$. Thus $\{H(x_k,y_k)\}_{k\in M}$ converges to $H(w,t)$. Also, $\{H(x_n,y_n)\}_{n\in\bbn\setminus M}$ is either finite or converges to $H(w,t)$. Hence $\{H(x_n,y_n)\}_{n\in\bbn}$ converges to $H(w,t)$. Therefore $H$ is continuous.

Let $\gamma(s)=H(s,1)$ and observe that $\gamma\in\Omega(J_{n-1}(X),e)$.
\end{proof}

\begin{lemma}\label{dimensionlower}
For $n\geq 2$, every loop $\alpha\in\Omega(J_n(X),e)$ is homotopic to a loop $\gamma\in\Omega(J_{n-1}(X),e)$.
\end{lemma}

\begin{proof}
For an arbitrary loop $\alpha\in\Omega(\jnx,e)$, we apply Lemma \ref{alltype1removal} to obtain a loop $\beta\in\Omega(\jnx,e)$, homotopic to $\alpha$, which satisfies the conditions of Lemma \ref{alltype2removal}. Then by Lemma \ref{alltype2removal}, $\beta$ is homotopic to loop $\gamma\in\Omega(J_{n-1}(X),e)$.
\end{proof}

Finally, we are able to complete the proof of Theorem \ref{surjectivitytheorem}.

\begin{proof}[Proof of Theorem \ref{surjectivitytheorem}]
Given a loop $\alpha\in \Omega(\jx,e)$, we have $Im(\alpha)\subseteq \jnx$ for some $n\in\bbn$ by Proposition \ref{sequentiallycompactprop}. Lemma \ref{agreeingtopologieslemma} ensures that $\alpha$ is continuous with respect to the quotient topology on $\jnx$. Repeated application of Lemma \ref{dimensionlower} implies that there exists $\beta\in\Omega(J_1(X),e)$ with $\alpha\simeq\beta$ in $\jx$. Thus $\sigma_{\#}([\beta])=[\alpha]$.
\end{proof}

We consider some immediate consequences of Theorem \ref{surjectivitytheorem}, which will be used in Section \ref{computationsection}. The Hurewicz homomorphism $h_X:\pi_1(X,e)\to H_1(X)$ is the abelianization map, $\pi_1(\jx,e)$ is abelian, and $\sigma_{\#}:\pi_1(X,e)\to \pi_1(\jx,e)$ is surjective. Hence, we have the following corollary.

\begin{corollary}\label{rhocorollary}
If $X$ is a path-connected Hausdorff space and $h_X:\pi_1(X,e)\to H_1(X)$ is the Hurewicz homomorphism, then there exists a unique, natural, surjective homomorphism $\rho:H_1(X)\to \pi_1(\jx,e)$ such that $\rho\circ h_X=\sigma_{\#}$.
\end{corollary}

The naturality of the maps $\sigma$ gives $\sigma_{\#}\circ f_{\#}=J(f)_{\#}\circ\sigma_{\#}$ for any based map $f:(X,e_1)\to (Y,e_2)$. Applying Theorem \ref{surjectivitytheorem} gives the following.

\begin{corollary}\label{inducedsurjectioncorollary}
If $Y$ is path-connected and Hausdorff and the map $f:(X,e_1)\to (Y,e_2)$ induces a surjection on $\pi_1$, then so does the homomorphism $J(f):(J(X),e_1)\to (J(Y),e_2)$.
\end{corollary}

\begin{corollary}\label{monoidisomorphism}
If $M$ is a path-connected, Hausdorff pre-$\Delta$-monoid, then $\sigma_{\#}:\pi_1(M,e)\to \pi_1(J(M),e)$ is an isomorphism.
\end{corollary}

\begin{proof}
The identity map $id:M\to M$ induces a continuous homomorphism $\wt{id}:J(M)\to M$ such that $\wt{id}\circ\sigma=id$. Since $\sigma$ is a section, $\sigma_{\#}$ is injective. Surjectivity is given by Theorem \ref{surjectivitytheorem}. 
\end{proof}

Hence, iterated application of the James reduced product immediately stabilizes for $\pi_1$.

\begin{remark}
Using the surjectivity of $\sigma_{\#}$ or $\rho$, known cardinality results for fundamental groups \cite{Pawlikowski,Shelah} or first singular homology \cite{ConnerCorsionH1} imply that if $X$ is a Peano continuum, then either (1) $X$ is semilocally simply connected and $\pi_1(\jx,e)$ is a finitely generated abelian group or (2) $X$ is not semilocally simply connected and $\pi_1(\jx,e)$ is uncountable.
\end{remark}

\subsection{Some Computations}\label{computationsection}

We begin by considering the case when $X$ is well-pointed at its basepoint $e\in X$, i.e. when the inclusion $\{e\}\to X$ is a cofibration. Although, we could argue that the proof in \cite[Chapter VII.2]{WhiteheadEOH} applies, we provide a simple and direct argument that the epimorphism $\rho:H_1(X)\to \pi_1(\jx,e)$ from Corollary \ref{rhocorollary} is an isomorphism. Let $\lambda:X\to \Omega^{\ast}(\Sigma X)$ be the unit map of the loop-suspension adjunction composed with the based-homotopy equivalence $\Omega(\Sigma X)\cong \Omega^{\ast}(\Sigma X)$ to the Moore loop space. Let $\wt{\lambda}:J(X)\to \Omega^{\ast}(\Sigma X)$ be the induced continuous homomorphism such that $\wt{\lambda}\circ\sigma=\lambda$.

\begin{theorem}\label{wellpointedisomorphismtheorem}
If $X$ is path-connected, Hausdorff, and well-pointed at its basepoint $e\in X$, then $\rho:H_1(X)\to \pi_1(\jx,e)$ is an isomorphism.
\end{theorem}

\begin{proof}
Since $X$ is well-pointed, $\Sigma X$ is homotopy equivalent to the unreduced suspension on $X$ and it follows that the suspension homomorphism $\Sigma:H_1(X)\to H_2(\Sigma X)$ is an isomorphism. The injectivity of $\rho$ may be verified directly from the following commutative diagram where $h_1,h_2$ are the Hurewicz homomorphisms. Note that the outermost square commutes by the compatibility of the suspension operations and Hurewicz homomorphisms. The surjectivity of $h_1$ then gives the commutativity of the bottom square.
\[\xymatrix{
    & \pi_1(X) \ar[dr]^-{\lambda_{\#}} \ar@{->>}[dl]_-{h_1} \ar[d]_-{\sigma_{\#}} & \\
 H_1(X) \ar[dr]_-{\Sigma}^{\cong} \ar[r]_-{\rho} & \pi_1(\jx)  \ar[r]^-{\wt{\lambda}_{\#}}  & \pi_2(\Sigma X) \ar[dl]^-{h_2}\\
& H_2(\Sigma X)
}\]
\end{proof}

\begin{example}\label{heexampleotherpoint}
It is known that $H_1(\bbh)\cong \bbz^{\bbn}\oplus (\bbz^{\bbn}/\oplus_{\bbn}\bbz)$ where the second summand combines \cite[VII.42 Exercise 7]{Fuchs} with the Eda-Kawamura representation in \cite{Edasingularonedim16}. If $e\in\bbh$ is any point other than $b_0$, write $J(\bbh,e)$ for the James reduced product on $(\bbh,e)$. Since $\bbh$ is locally contractible at $e$, we have canonical isomorphism $\rho:H_1(\bbh)\to\pi_1(J(\bbh,e),e)$. Although $J(\bbh,e)$ is transfinitely $\pi_1$-commutative at its identity $e$ (since it is locally contractible at $e$), it is not transfinitely $\pi_1$-commutative at $b_0\in\bbh\subseteq J(\bbh,e)$. For if it were, the loop $\prod_{n=1}^{\infty}(\ell_n\cdot\ell_{n+1}\cdot\ell_{n}^{-}\cdot\ell_{n+1}^{-})\in \Omega(J(\bbh,e),b_0)$ would represent the trivial element in $\pi_1(J(\bbh,e),b_0)$. However, this loop is not null-homologous in $\bbh$ by Eda's $0$-form Lemma \cite[Lemma 3.6]{Edasingularonedim16}.
\end{example}

We contrast the previous example by using transfinite $\pi_1$-commutativity to compute $\pi_1(J(\bbh),b_0)$. Recall that $\mathbb{T}=\prod_{n\in\bbn}S^1$ is the infinite torus and that the embedding $\eta:\bbh\to\bbt$ induces the surjection $\eta_{\#}:\pi_1(\bbh,b_0)\to \bbz^{\bbn}$. Since $\bbt$ is a topological group with identity element $x_0$, there is a unique continuous monoid homomorphism $\wt{\eta}:J(\bbh)\to\bbt$ such that $\wt{\eta}\circ\sigma=\eta$.

\begin{theorem}\label{jhtheorem}
$\wt{\eta}_{\#}:\pi_1(J(\bbh),b_0)\to \pi_1(\bbt,x_0)$ is an isomorphism. In particular, $\pi_1(J(\bbh),b_0)\cong \bbz^{\bbn}$.
\end{theorem}

\begin{proof}
By Theorem \ref{surjectivitytheorem}, the inclusion $\sigma:\bbh\to J(\bbh)$ induces a surjection on $\pi_1$ and by Corollary \ref{factorizationlemma}, there exists a unique homomorphism $g:\pi_1(\bbt,x_0)\to \pi_1(J(\bbh),b_0)$ such that $g\circ \eta_{\#}=\sigma_{\#}$.
\[\xymatrix{
& \pi_1(\bbh,b_0) \ar@{->>}[dl]_-{\eta_{\#}} \ar@{->>}[d]_-{\sigma_{\#}} \ar@{->>}[dr]^-{\eta_{\#}} \\
\pi_1(\bbt,x_0) \ar[r]_-{g} & \pi_1(J(\bbh),b_0) \ar[r]_-{\wt{\eta}_{\#}} & \pi_1(\bbt,x_0)
}\]
Since $\wt{\eta}_{\#}\circ g\circ \eta_{\#}=\wt{\eta}_{\#}\circ\sigma_{\#}=\eta_{\#}$ where $\eta_{\#}$ is surjective, we have $\wt{\eta}_{\#}\circ g=id$. Additionally, since $\sigma_{\#}=g\circ \eta_{\#}=g\circ \wt{\eta}_{\#}\circ\sigma_{\#}$ and $\sigma_{\#}$ is surjective, we have $g\circ \wt{\eta}_{\#}=id$. Thus $g$ and $\wt{\eta}_{\#}$ are inverses.
\end{proof}

\begin{remark}
Since $\pi_1(J(\bbh),b_0)\cong \bbz^{\bbn}$ is cotorsion-free, it follows from Theorem \ref{cotorsionfreemainthm} that $J(\bbh)$ is transfinitely $\pi_1$-commutative and has well-defined transfinite $\pi_1$-products at all of its points. This provides an example of a monoid, which is not a group, satisfying the conclusion of Theorem \ref{grouptheorem}.
\end{remark}

\begin{remark}
The previous computations for $(\bbh,b_0)$ and $(\bbh,e)$ with $e\neq b_0$ indicate that given two points $e_1,e_2\in X$, the groups $\pi_1(J(X,e_1),e_1)$ and $\pi_1(J(X,e_2),e_2)$ need not be isomorphic. In particular, if $e\neq b_0$, then $\pi_1(J(\bbh),e)\cong H_1(\bbh)$ is first singular homology and $\pi_1(J(\bbh),b_0)\cong \check{H}_1(\bbh)$ is first \v{C}ech singular homology group.
\end{remark}

\begin{example}[Double Hawaiian Earring]\label{doubleheexample}
In $\bbr^2$, set $\bbh_1=\bbh$, $\bbh_2=\{(-x-1,y)\mid (x,y)\in\bbh\}$, and $A=[-1,0]\times \{0\}$. Then $X=\bbh_1\cup \bbh_2\cup A$ consists of the line segment $A$ with a copy of $\bbh$ attached at each endpoint (See Figure \ref{hearc}). We take the origin $e=b_0$ to be the basepoint of $X$ and let $D$ be the smallest closed disk containing $\bbh_1$. Let $i:\bbh_1\to X$ be the inclusion and $r:X\to \bbh_2\cup A$ be the map collapsing $\bbh_1$ to $e$. Since $\bbh_1$ and $\bbh_2\cup A$ are both based retracts of $X$, the induced homomorphisms $i_{\#}$, $J(i)_{\#}$ in the diagram below are injective and, likewise, the homomorphisms $r_{\#}$, $J(r)_{\#}$ are surjective.
\[\xymatrix{
1 \ar[r] & \pi_1(\bbh_1,e) \ar@{->>}[d]_-{\sigma_{\#}}  \ar[r]^-{i_{\#}} & \pi_1(X,e) \ar@{->>}[d]_-{\sigma_{\#}} \ar@{->>}[r]^-{r_{\#}} & \pi_1(\bbh_2\cup A,e) \ar@{->>}[d]_-{\sigma_{\#}}  \ar[r] & 1\\
1 \ar[r] & \pi_1(J(\bbh_1),e) \ar[r]_-{J(i)_{\#}} & \pi_1(\jx,e) \ar@{->>}[r]_-{J(r)_{\#}} & \pi_1(J(\bbh_2\cup A),e) \ar[r] & 1
}\]
Since the composition $r_{\#}\circ i_{\#}$ is trivial and the vertical maps are surjective, it is clear that $J(r)_{\#}\circ J(i)_{\#}$ is trivial. We show the lower sequence is exact. Suppose $\alpha\in \Omega(X,e)$ such that $J(r)_{\#}(\sigma_{\#}([\alpha]))=[\sigma\circ r\circ\alpha]=1$ in $\pi_1(J(\bbh_2\cup A),e)$. Since $X=(\bbh_1,e)\vee (\bbh_2\cup A,e)$ where $\bbh_2\cup A$ is locally simply connected at $e$, we may write $\alpha=a_1\cdot b_1\cdot a_2\cdot b_2\cdots a_m\cdot b_m$ where $a_k\in \Omega(\bbh_1,e)$ and $b_k\in \Omega(\bbh_2\cup A,e)$. However, since $\bbh_2\cup A$ is well-pointed at $e$, the homomorphism $\sigma_{\#}:\pi_1(\bbh_2\cup A)\to \pi_1(J(\bbh_2\cup A),e)$ is abelianization. Therefore, $[r\circ\alpha]=\left[\prod_{k=1}^{m}b_k\right]$ lies in the commutator subgroup of $\pi_1(\bbh_2\cup A,e)$. Since $r_{\#}$ is a retraction, $\left[\prod_{k=1}^{m}b_k\right]$ also lies in the commutator subgroup of $\pi_1(X,e)$. Since $\pi_1(\jx,e)$ is abelian, we have $\left[\sigma\circ \prod_{k=1}^{m}b_k\right]=1$ in $\pi_1(J(X),e)$ and therefore, \[J(i)_{\#}\left(\left[\sigma\circ\prod_{k=1}^{m}a_k\right]\right)=\prod_{k=1}^{m}[\sigma\circ a_k]=\prod_{k=1}^{m}[\sigma\circ a_k]\prod_{k=1}^{m}[\sigma\circ b_k]=\sigma_{\#}([\alpha]).\]
This proves that $\ker(J(r)_{\#})\leq Im(J(i)_{\#})$.

Since the lower sequence is exact and $\bbh_1$ is a retract of $X$, the homomorphism $J(i)_{\#}$ splits to give $\pi_1(\jx,e)\cong \pi_1(J(\bbh_1),e)\oplus \pi_1(J(\bbh_2\cup A),e)$. Recall that we computed $\pi_1(J(\bbh_1),e)\cong \check{H}_1(\bbh)\cong \bbz^{\bbn}$ in Theorem \ref{jhtheorem} and since $\bbh_2\cup A$ is well-pointed at $e$ and (freely) homotopic to $\bbh_2$, we have $\pi_1(J(\bbh_2\cup A),e)\cong H_1(\bbh_2\cup A)\cong H_1(\bbh_2)\cong \bbz^{\bbn}\oplus (\bbz^{\bbn}/\oplus_{\bbn}\bbz)$. We find that it is most intuitive to understand $\pi_1(\jx,e)$ by the isomorphism $\pi_1(\jx,e)\cong \check{H}_1(\bbh)\oplus H_1(\bbh)$. The first summand is \v{C}ech homology since $\jx$ is transfinitely $\pi_1$-commutative at $e$. Since the wild point $(-1,0)\in \bbh_2$ is not used as the basepoint to construct $\jx$, one may show, as in Example \ref{heexampleotherpoint}, that $\jx$ is not transfinitely $\pi_1$-commutative at $(-1,0)$. This explains why the second summand is ordinary singular homology, i.e. the finitary abelianization of $\pi_1(\bbh,b_0)$. Abstractly, we have an isomorphism $\pi_1(\jx,e)\cong \bbz^{\bbn}\oplus (\bbz^{\bbn}/\oplus_{\bbn}\bbz)$ since $\bbz^{\bbn}\oplus \bbz^{\bbn}\cong \bbz^{\bbn}$.

Finally, by taking an infinite concatenation of commutators in $\Omega(\bbh_2,(-1,0))$ as in Example \ref{heexampleotherpoint}, we see that $J(X)$ does not have well-defined infinite $\pi_1$-products at $(-1,0)$. On the other hand, consider the open neighborhood $U=\bbh_1\cup (-1/2,0]$ of $e$ in $X$. Since $q_{n}^{-1}(J_n(U))=U^n$ in $X^n$, a straightforward argument shows that $J(U)$ embeds onto an open subspace of $J(X)$. Since $U\simeq \bbh_1$ rel. basepoint, we have that $\pi_1(J(U),e)\cong \pi_1(J(\bbh_1),e)\cong \bbz^{\bbn}$ is cotorsion-free. Thus, by Theorem \ref{cotorsionfreemainthm}, $J(U)$ has well-defined transfinite $\pi_1$-products at all of its points. Since there is a path-connected open neighborhood of $e$ in $J(X)$ with well-defined transfinite $\pi_1$-products at $e$, $J(X)$ also has well-defined transfinite $\pi_1$-products at $e$. 
\end{example}

\begin{figure}[H]
\centering \includegraphics[height=1.5in]{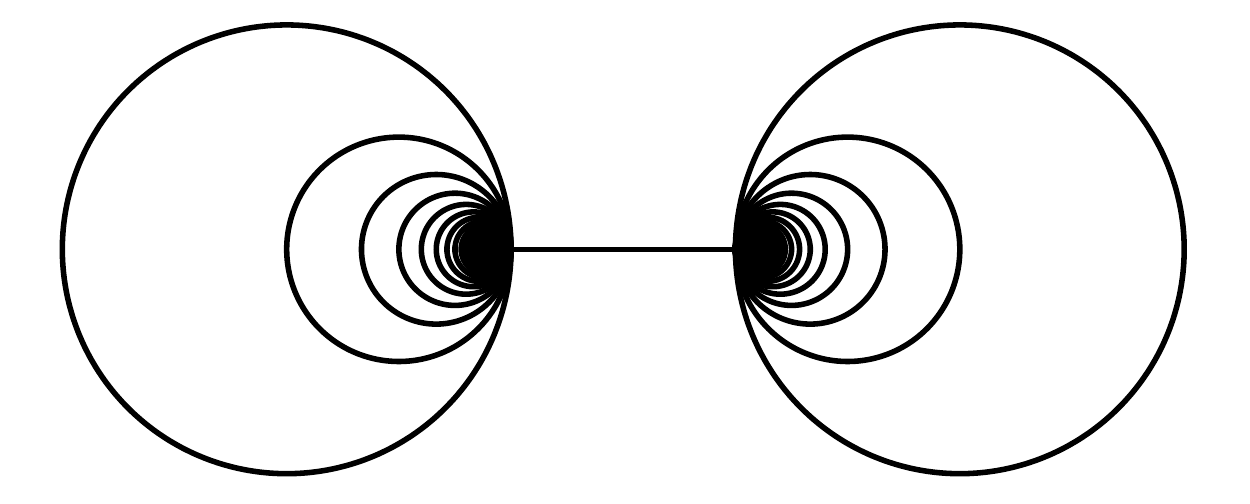}
\caption{\label{hearc}The space $X=\bbh_1\cup \bbh_2\cup A$}
\end{figure}

\begin{remark}[Transfinite Abelianization]\label{infinitaryabelianization}
The above examples serve to show that $\pi_1(\jx,e)$ is an abelian quotient of $\pi_1(X,e)$ in which infinite products of commutators of loops \textit{based only at} $e$ are trivialized. In this sense, $\pi_1(\jx,e)$ appears to behave as an ``infinitary abelianization" of $\pi_1(X,e)$ at $e$. In the case of $X=\bbh$, $\pi_1(J(\bbh),b_0)\cong \bbz^{\bbn}$ agrees with the notion of ``strong abelianization" of $\pioneh$ defined in \cite{CC,Edafreesigmaproducts}. Furthermore, the examples above justify transfinite $\pi_1$-commutativity (Definition \ref{deftransfinitecommut}) as a point-wise property since a space might only have this property at some of its points. This allows one to transfinitely abelianize $\pi_1(X,e)$ at any non-empty subset $A\subseteq X$: \textit{the transfinite abelianization of $\pi_1(X,e)$ at $A\subseteq X$} is the quotient $\pi_1(X,e)/C_{\tau}(A)$ where $C_{\tau}(A)$ is generated by homotopy classes of loops $\beta\cdot \prod_{\tau}\alpha_n\cdot \left(\prod_{\tau}\alpha_{\phi(n)}\right)^{-}\cdot \beta^{-}$ where $\beta:I\to X$ is a path from $e$ to $a\in A$, $\{\alpha_n\}$ is null at $a$, and $\phi:\bbn\to\bbn$ is a bijection. Note that $C_{\tau}(A)$ contains the ordinary commutator subgroup of $\pi_1(X,e)$ as long as $A\neq \emptyset$. Hence, $\pi_1(X,e)/C_{\tau}(A)$ is commutative in the usual sense. By Theorem \ref{monoidtransfinitecommutethm}, $\pi_1(J(X),e)$ is a quotient of the transfinite abelianization of $\pi_1(X,e)$ at $\{e\}$. These two groups are equivalent for all examples considered in this paper. Although the authors do not know if $\pi_1(J(X),e)$ and$\pi_1(X,e)/C_{\tau}(\{e\})$ are always isomorphic, it is expected that such equivalence would occur often (see Problem \ref{endproblem} below).
\end{remark}

Recall that there exists a natural epimorphism $\rho:H_1(X)\to \pi_1(\jx,e)$ induced by the Hurewicz map. Additionally there exists a canonical homomorphism $\Upsilon:H_1(X)\to \check{H}_1(X)$ to the first \v{C}ech singular homology group.

\begin{theorem}\label{kappahomomorphism}
If $X$ is a Peano continuum, then there is a natural surjective homomorphism $\kappa:\pi_1(\jx,e)\to \check{H}_1(X)$ such that $\kappa\circ \rho=\Upsilon$. Moreover, if $\Upsilon$ splits, then so does $\kappa$.
\end{theorem}

\begin{proof}
Consider $X$ as the inverse limit $\varprojlim_{n}(X_n,r_{n+1,n})$ of compact polyhedra $X_n$ so that the inverse system $(X_n,r_{n+1,n})$ is an $HPol$-expansion for $X$ (see \cite{BF13,MS82} for details). Let $r_n:X\to X_n$ denote the projection maps. Since the compact metric space $X$ is locally path-connected, we may choose the spaces $X_n$ so that each homomorphism $(r_n)_{\#}:\pi_1(X,e)\to \pi_1(X_n,x_n)$ is surjective \cite{BF13}. For each $n\in\bbn$, we have the following commutative square where the vertical morphisms are the surjections induced by the Hurewicz homomorphism.
\[\xymatrix{
H_1(X) \ar@{->>}[r]^-{r_{n\ast}} \ar@{->>}[d]_-{\rho} & H_1(X_n) \ar[d]_-{\cong}^-{\rho_n} \\
\pi_1(\jx,e) \ar@{-->>}[ur]^-{\kappa_n} \ar@{->>}[r]_-{J(r_n)_{\#}} & \pi_1(J(X_n),p_n(e))
}\]
Since $r_n$ induces a surjection on $\pi_1$, both horizontal morphisms are surjective. Moreover, since $X_n$ is well-pointed at all of it's points, $\rho_n$ is an isomorphism (recall Theorem \ref{wellpointedisomorphismtheorem}). Thus $\kappa_n=(\rho_{n}^{-1}\circ J(r_n)_{\#}):\pi_1(\jx,e)\to H_1(X_n)$ is a surjective homomorphism satisfying $\kappa_n\circ \rho=r_{n\ast}$. The naturality of the maps $\rho_n$ give $\rho_n\circ (r_{n+1,n})_{\ast}\circ \kappa_{n+1}=J(r_{n+1,n})_{\#}\circ \rho_{n+1}\circ \kappa_{n+1}=J(r_{n+1,n})_{\#}\circ J(r_{n+1})_{\#}=J(r_n)_{\#}$, which shows that $(r_{n+1,n})_{\ast}\circ \kappa_{n+1}=\rho_{n}^{-1}\circ J(r_{n})_{\#}=\kappa_n$. Therefore, we have a natural induced homomorphism $\kappa:\pi_1(\jx,e)\to \varprojlim_{n}H_1(X_n)=\check{H}_1(X_n)$ to the inverse limit. Our construction of $\kappa$ ensures that $\kappa\circ \rho=\Upsilon$. Since $X$ is a Peano continuum, $\Upsilon :H_1(X)\to \check{H}_1(X)$ is surjective \cite{EKSingularSurj}. Therefore, $\kappa$ is surjective. For the final statement, note that if $\Upsilon$ has a section $s:\check{H}_1(X)\to H_1(X)$, then $\rho\circ s$ is a section of $\kappa$.
\end{proof}

It is known that when $X$ is a one-dimensional Peano continuum, the homomorphism $\Upsilon$ splits \cite{Edasingularonedim16}. Hence, we have the following corollary.

\begin{corollary}
If $X$ is a one-dimensional Peano continuum, then $\check{H}_1(X)$ is a summand of $\pi_1(\jx,e)$.
\end{corollary}

\begin{example}[Shrinking Wedges]
Let $\{X_k\}_{k\in\bbn}$ be a sequence of based, connected, compact CW-complexes and $X=\wt{\bigvee}_{k}X_k$ be the ``shrinking wedge" of this sequence with wedge point $e\in X$. Then $X$ is a Peano continuum and it is well-known that $\check{H}_1(X)\cong \prod_{k\in\bbn}H_1(X_k)$ \cite{Edafreesigmaproducts}. We show that $\pi_1(J(X),e)\cong \prod_{k\in\bbn}H_1(X_k)$ by showing the epimorphism $\kappa$ from Theorem \ref{kappahomomorphism} is injective. Let $r_k:X\to X_k$ be the canonical projections and suppose $\alpha\in \Omega(J(X),e)$ is a non-constant loop such that $\kappa([\alpha])=0$. We may assume $Im(\alpha)\subseteq X$, which gives that $r_k\circ\alpha$ is null-homologous in $X_k$ for all $k\in\bbn$. Applying Lemma \ref{generalpermutationlemma} to $(J(X),e)$ and $\mci(\alpha^{-1}(e))$, we see that $\alpha$ is homotopic in $J(X)$ to an infinite concatenation $\beta=\prod_{k=1}^{\infty}\beta_k$ where $\beta_k$ is a null-homologous loop in $X_k$. Since $\pi_1(X,e)$ canonically inject's into $\check{\pi}_1(X,e)$ \cite{MM}, it follows that $X$ is homotopically Hausdorff \cite[Lemma 2.1]{CMRZZ08}. Hence, we may replace each $\beta_k$ with a homotopic finite concatenation of commutator loops in $X_k$ without changing the homotopy class of $\beta$. Now, we may perform an infinite shuffle identical to that in the proof of (3) $\Rightarrow$ (1) in Theorem \ref{transfinitecharacteriationtheorem} to see that $\beta$ is homotopic in $\jx$ to an infinite concatenation $\gamma$ of consecutive inverse pairs in $X$. Such a loop $\gamma$ is null-homotopic in $X$ and hence $[\alpha]=[\gamma]=0$ in $\pi_1(\jx,e)$.
\end{example}

Our final two examples include spaces with trivial first \v{C}ech homology.

\begin{example}[Harmonic Archipelago]\label{exampleha}
The harmonic archipelago $\bbh\bba$ is the space obtained by attaching a 2-cell $e_n$ to $\bbh$ along the loop $\ell_{n}\cdot\ell_{n+1}^{-}$ for all $n\in\bbn$. Although $\pi_1(\bbh\bba,b_0)$ is uncountable and locally free \cite{hojkaHA}, every loop $\alpha\in\Omega(\bbh\bba,b_0)$ is homotopic to a loop in every neighborhood of $b_0$. Let $A_1=\bbh$ and $A_m=\bbh\cup e_1\cup e_2\cup\cdots \cup e_{m-1}$ so that $\{A_m\}$ is a $k_{\omega}$-decomposition for $\bbh\bba$. By Lemma \ref{komegalemma}, $J(\bbh\bba)\cong \varinjlim_{m} J(A_m)$ and $\{J_m(A_m)\}$ is a $k_{\omega}$-decomposition of $J(\bbh\bba)$. In particular, every loop and null-homotopy in $J(\bbh\bba)$ must have image in $J(A_m)$ for some $m\in\bbn$. Let $Y=J(\bbh)\cup\bbh\bba$, that is, the space obtained by attaching the 2-cells $e_n$ to $J(\bbh)$ by the attaching maps $\sigma\circ(\ell_{n}\cdot\ell_{n+1}^{-}):I\to \bbh\to J(\bbh)$. The map $\sigma:\bbh\bba\to J(\bbh\bba)$ factors as $\sigma=i_1\circ i_2$ for inclusions $i_2:\bbh\bba\to Y$ and $i_1:Y\to J(\bbh\bba)$. Since the induced homomorphism $\sigma_{\#}:\pi_1(\bbh\bba,b_0)\to \pi_1(J(\bbh\bba),b_0)$ is surjective by Theorem \ref{surjectivitytheorem}, $(i_1)_{\#}:\pi_1(Y,b_0)\to \pi_1(J(\bbh\bba),b_0)$ is surjective. We will show that $(i_1)_{\#}$ is an isomorphism. 

First, let $r_m:\bbh\to \bbh_{\geq m}$ be the retraction satisfying $f\circ\ell_i=\ell_m$ for $1\leq i\leq m-1$ and consider the following induced diagram where $i_3,i_4,i_5,i_6$ are the indicated inclusions.
\[\xymatrix{
\pioneh \ar[d]_-{(r_m)_{\#}} \ar@{->>}@/^4pc/[drr]^-{(i_5)_{\#}} \ar@{->>}[r]^-{\sigma_{\#}} & \pi_1(J(\bbh),b_0) \ar[d]_-{J(r_m)_{\#}} \ar@{->>}[dr]^-{(i_4)_{\#}}\\
\pi_1(\bbh_{\geq m},b_0) \ar@/_2pc/[rr]_-{(i_6)_{\#}} \ar@{->>}[r]_-{\sigma_{\#}} & \pi_1(J(\bbh_{\geq m}),b_0) \ar[r]_-{(i_3)_{\#} } & \pi_1(Y,b_0)
}\]
The top and bottom triangles and the square obviously commute. The outermost triangle commutes since $i_6\circ r_m\simeq i_5$ for any $m\in\bbn$. The last triangle commutes, i.e. $(i_3)_{\#}\circ J(r_m)_{\#}=(i_4)_{\#}$, by the surjectivity of the top $\sigma_{\#}$ and the commutativity of the other subdiagrams. Note that since $(i_4)_{\#}$ is surjective by the construction of $Y$, $(i_3)_{\#}$ is also surjective for all $m\in \bbn$.

Now suppose $\beta\in \Omega(J(\bbh_{\geq m}),b_0)$ is such that $i_1\circ i_3\circ\beta$ is null-homotopic in $J(\bbh\bba)$ Then $i_1\circ i_3\circ\beta$ is null-homotopic in $ J(A_m)$ for some $m\in\bbn$. We have the following commutative square where $i_7,i_8$ are the indicated inclusions.
\[\xymatrix{
\pi_1(J(\bbh_{\geq m})) \ar[r]^-{J(i_7)_{\#}}_-{\cong} \ar[d]_-{(i_3)_{\#}} & \pi_1(J(A_m),b_0) \ar[d]^-{J(i_8)_{\#}} \\
\pi_1(Y,b_0) \ar[r]_-{(i_1)_{\#}} & \pi_1(J(\bbh\bba),b_0)
}\]
Since $i_7:\bbh_{\geq m}\to A_m$ is a based homotopy equivalence, Lemma \ref{homotopyequivalenlemma} gives that $J(i_7)_{\#}$ is an isomorphism. Therefore, since $J(i_7)_{\#}([\beta])=1$, we have $[\beta]=1$ and thus $[i_3\circ\beta]=1$. Since $(i_3)_{\#}:\pi_1(\bbh_{\geq m},b_0)\to\pi_1(Y,b_0)$ is surjective, the injectivity of $(i_1)_{\#}$ follows.

By Theorem \ref{jhtheorem}, we may identify $\pi_1(J(\bbh),b_0)$ with $\bbz^{\bbn}$ where $[\sigma\circ\ell_n]$ corresponds to the unit vector $\mathbf{e}_n\in \bbz^{\bbn}$. Therefore, $\pi_1(Y,b_0)$ is isomorphic to the quotient $\bbz^{\bbn}/N$ where $N=\langle \mathbf{e}_n-\mathbf{e}_{n+1}\mid n\in\bbn\rangle$ for all $j,k\in\bbn$. It is straightforward to check that there is an isomorphism $\bbz^{\bbn}/\oplus_{\bbn}\bbz\to\bbz^{\bbn}/N$ induced by the homomorphism $\bbz^{\bbn}\to \bbz^{\bbn}$, $(a_1,a_2,a_3,\dots)\mapsto (a_1,a_2-a_1,a_3-a_2,\dots)$. We conclude that $\pi_1(J(\bbh\bba),b_0)\cong \pi_1(Y,b_0)\cong \bbz^{\bbn}/\oplus_{\bbn}\bbz$. Additionally, since $\bbh\bba$ is not homotopically Hausdorff at $b_0$ and $\sigma_{\#}:\pi_1(\bbh\bba,b_0)\to \pi_1(J(\bbh\bba),b_0)$ is non-trivial, $J(\bbh\bba)$ is not homotopically Hausdorff at $b_0$. It follows from Proposition \ref{hhausprop} that $J(\bbh\bba)$ is not homotopically Hausdorff at any of its points.
\end{example}

\begin{example}[Griffiths Twin Cone]
Let $O=\{2k-1\mid k\in\bbn\}$ and $E=\{2k\mid k\in\bbn\}$ so that $\bbh_{O}$ and $\bbh_{E}$ are the sub-Hawaiian earrings of odd and even indexed circles respectively. Let $C\bbh_{O}$ be the cone over $\bbh_{O}$ in $\bbr^3$ with vertex $(0,0,1)$ and $C\bbh_{E}$ be the cone over $\bbh_{E}$ in $\bbr^3$ with vertex $(0,0,-1)$. The space $\bbg=C\bbh_{O}\cup C\bbh_{E}$ is often called the Griffiths Twin Cone and is known to have uncountable fundamental group despite being the one-point union of two contractible spaces \cite{Griffiths}. Let $i:\bbh\to\bbg$ be the inclusion and consider the following diagram where $g$ is the isomorphism from Theorem \ref{jhtheorem}. 
\[\xymatrix{
\pioneh \ar@{->>}@/^2.1pc/[rr]^-{\sigma_{\#}} \ar[d]_-{i_{\#}} \ar[r]^-{\eta_{\#}} & \bbz^{\bbn} \ar[r]^-{g}_-{\cong} & \pi_1(J(\bbh),b_0) \ar[d]^-{J(i)_{\#}} \\
\pi_1(\bbg,b_0) \ar@{->>}[rr]_-{\sigma_{\#}} && \pi_1(J(\bbg),b_0)
}\]
The homomorphism $i_{\#}$ is surjective by the van Kampen Theorem and thus $J(i)_{\#}$ is surjective by Corollary \ref{inducedsurjectioncorollary}. Consider an arbitrary element $(a_1,a_2,a_3,a_4,\dots)\in\bbz^{\bbn}$ and define the corresponding loop $\alpha=\left(\prod_{k=1}^{\infty}\ell_{2k-1}^{a_{2k-1}}\right)\cdot \left(\prod_{k=1}^{\infty}\ell_{2k}^{a_{2k}}\right)$ in $\bbh$. Then $\eta_{\#}([\alpha])=(a_1,a_2,a_3,a_4,\dots)$. However, $i_{\#}([\alpha])=1$ in $\pi_1(\bbg,b_0)$ and thus $J(i)_{\#}\circ g((a_1,a_2,a_3,a_4,\dots))=\sigma_{\#}(i_{\#}([\alpha]))=1$. This proves that the surjection $J(i)_{\#}\circ g$ is trivial. We conclude that $\pi_1(J(\bbg),b_0)$ is the trivial group.
\end{example}

\begin{remark}
It remains an open question whether or not $\pi_1(\bbh\bba,b_0)$ and $\pi_1(\bbg,b_0)$ are isomorphic groups \cite[Problem 1.1]{KarimovRepovs}. In the previous two examples, we have shown that $\pi_1(J(\bbh\bba),b_0)\cong \bbz^{\bbn}/\oplus_{\bbn}\bbz$ and $\pi_1(J(\bbg),b_0)=1$. We encourage the reader to check that $\pi_1(J(\bbh\bba),b_0)$ and $\pi_1(J(\bbg),b_0)$ are isomorphic to the respective transfinite abelianizations of $\pi_1(\bbh\bba,b_0)$ and $\pi_1(\bbg,b_0)$ at $\{b_0\}$ (in the sense of Remark \ref{infinitaryabelianization}). In this way, we have identified a distinction between the natural infinitary structure each group inherits from the loop space. Since the referenced open problem only involves the underlying groups and not these infinitary operations, the fact that $\pi_1(J(\bbh\bba),b_0)\ncong \pi_1(J(\bbg),b_0)$ does not distinguish the group-isomorphism types of $\pi_1(\bbh\bba,b_0)$ and $\pi_1(\bbg,b_0)$.
\end{remark}

Recall that the natural map $\lambda:X\to \Omega^{\ast}(\Sigma X)$ to the Moore loop space induces a continuous homomorphism $\wt{\lambda}:J(X)\to \Omega^{\ast}(\Sigma X)$ such that $\wt{\lambda}\circ\sigma=\lambda$. Additionally, by Remark \ref{infinitaryabelianization}, there is a canonical surjection $\Lambda:\pi_1(X,e)/C_{\tau}(\{e\})\to \pi_1(J(X),e)$ on the transfinite abelianization at $\{e\}$ induced by $\sigma_{\#}$. Hence, we have the following commutative diagram.
\[\xymatrix{
& \pi_1(X,e) \ar@{->>}[dl] \ar@{->>}[d]_-{\sigma_{\#}} \ar[dr]^-{\lambda_{\#}}\\
\pi_1(X,e)/C_{\tau}(\{e\}) \ar@{->>}[r]_-{\Lambda} & \pi_1(J(X),e) \ar[r]_-{\wt{\lambda}_{\#}} & \pi_2(\Sigma X,e_0)
}\]
While $\wt{\lambda}_{\#}$ is an isomorphism on $\pi_1$ when $X$ is well-pointed (from the proof Theorem \ref{wellpointedisomorphismtheorem}), the authors find the results in this paper to be evidence that both $\wt{\lambda}_{\#}$ and $\Lambda$ are isomorphisms for all path-connected Hausdorff $X$.

\begin{problem}\label{endproblem}
For a path-connected Hausdorff space $X$ and point $e\in X$, consider the following three groups:
\begin{enumerate}
\item $\pi_1(J(X),e)$
\item $\pi_2(\Sigma X,e_0)$
\item $\pi_1(X,e)/C_{\tau}(\{e\})$.
\end{enumerate}
Must these groups all be naturally isomorphic to each other? 
\end{problem}

\end{document}